\documentclass[final,reqno,onefignum,onetabnum]{siamart220329}
\usepackage{times}
\usepackage{subdepth}
\usepackage{overpic}
\usepackage{amsmath,amsbsy,amssymb,mathdots}
\usepackage{booktabs,bm}
\usepackage{mathrsfs}
\usepackage{tikz,cite}
\usepackage{courier}
\usetikzlibrary{positioning}
\usetikzlibrary{patterns}
\usetikzlibrary{shapes,arrows}
\usepackage{tikz-qtree}
\usepackage[fleqn,tbtags]{mathtools}
\usepackage{amsfonts}
\usepackage{relsize}
\usepackage{xcolor,colortbl}
\usepackage{psfrag}
\usepackage{alltt}
\usepackage{arydshln}
\usepackage{enumerate}
\usepackage{epstopdf}
\usepackage{enumitem}
\usepackage{algorithm, algpseudocode}
\usepackage{algorithmicx}
\usepackage{subcaption}
\usepackage{blkarray}
\usepackage{pict2e}
\usepackage{comment}

\makeatletter
\renewcommand*\env@matrix[1][*\c@MaxMatrixCols c]{%
  \hskip -\arraycolsep
  \let\@ifnextchar\new@ifnextchar
  \array{#1}}
\makeatother

\DeclareMathOperator{\rank}{rank}
\newcommand{\erank}{\rank_\epsilon}
\DeclareMathOperator{\diag}{diag}

\newcolumntype{"}{@{\hskip\tabcolsep\vrule width 1pt\hskip\tabcolsep}}

\renewcommand{\arc}{\mathcal{A}}
\newcommand{\cluster}{\mathscr{I}}

\definecolor{myblue}{rgb}{0.36, 0.54, 0.66}
\newcommand{\tree}{\mathcal{T}}

\newcommand{\complex}{\mathbb{C}}
\newcommand{\twobytwo}[4]{\begin{bmatrix} #1 & #2 \\ #3 & #4 \end{bmatrix}}
\newcommand{\twobyone}[2]{\begin{bmatrix} #1 \\ #2 \end{bmatrix}}
\newcommand{\onebytwo}[2]{\begin{bmatrix} #1 & #2 \end{bmatrix}}
\newcommand{\order}{\mathcal{O}}

\DeclareMathOperator*{\argmin}{argmin}
\newcommand{\norm}[1]{\left\|#1\right\|}
\usepackage{tcolorbox}
\newcommand{\boxedtext}[1]{\begin{tcolorbox}[colback=white,colframe=black,width=\columnwidth,boxsep=5pt,arc=4pt]
  #1
\end{tcolorbox}}
\newcommand{\opentriangle}{%
  \raisebox{0.2pt}{\makebox[0.77778em]{%
    \setlength{\unitlength}{0.6em}%
    \linethickness{0.4pt}%
    \begin{picture}(1,1)
    \polygon(0,0)(1,0)(1,1)
    \end{picture}%
  }}%
}
\newcommand{\matrixsize}[1]{\textcolor{gray}{#1}}

\newcommand{\be}{\begin{equation}}
\newcommand{\ee}{\end{equation}}
\newcommand{\R}{\mathbb{R}}
\newcommand{\bigO}{{\mathcal{O}}}
\DeclareMathOperator{\sinc}{sinc}

\theoremstyle{remark}
\newtheorem{rmk}{Remark}

\title{Superfast direct inversion of the nonuniform discrete Fourier transform via hierarchically semi-separable least squares \thanks{Funding: This work is partly supported by the National Science Foundation grants DMS-2103317, DMS-2410045, DMS-2103317 and FRG award 1952777, the Carver Mead New Horizons Fund, and the U.S. Department of Energy, Office of Science, Office of Advanced Scientific Computing Research, Department of Energy Computational Science Graduate Fellowship under award number DE-SC0021110.} }
\author{
  Heather Wilber\thanks{Department of Applied Mathematics, University of Washington, Seattle, WA, 98105, United States (\email{hdw27@uw.edu}).}
\and
Ethan N. Epperly\thanks{Division of Computing and Mathematical Sciences, California Institute of Technology, Pasadena, CA, 91125, United States (\email{eepperly@caltech.edu}).}
\and
Alex H. Barnett\thanks{Center for Computational Mathematics, Flatiron Institute, New York, NY, 10010, United States (\email{abarnett@flatironinstitute.org}).}
}

\begin{document}
\maketitle

\begin{abstract} A direct solver is introduced for solving overdetermined linear systems involving nonuniform discrete Fourier transform matrices. Such matrices can be transformed into a Cauchy-like form that has hierarchical low rank structure. The rank structure of this matrix is explained, and it is shown that the ranks of the relevant submatrices grow only logarithmically with the number of columns of the matrix. A fast rank-structured hierarchical approximation method based on this analysis is developed, along with a hierarchical least-squares solver for these and related systems. This result is a direct method for inverting nonuniform discrete transforms with a complexity that is usually nearly linear with respect to the degrees of freedom in the problem.
  This solver is benchmarked against various iterative and direct solvers
  in the setting of inverting the one-dimensional type-II (or forward) transform,
  for a range of condition numbers and problem sizes (up to \smash{$4\times 10^6$ by $2\times 10^6$}). 
These experiments demonstrate that this method is especially useful for large  problems
  with multiple right-hand sides.

\end{abstract}



\section{Introduction} The nonuniform discrete Fourier transform (NUDFT) is a fundamental task in computational mathematics. In one dimension, the forward problem is to evaluate
\begin{equation} \label{eq:Vandermonde_sys}
 b_j = \sum_{k = 1}^{n} e^{-2 \pi i p_j w_k} x_k, \qquad 1 \leq j \leq m,
\end{equation}
where the  \textit{sample locations} \hbox{$0 \le p_m < \cdots < p_{1} < 1$}, the \emph{frequencies} \hbox{$0 \!\le\! w_1 \!<\! \cdots \!<\! w_{n} \!<\! n$}, and the \textit{coefficients} \hbox{$\{x_1,\ldots, x_n\}$} are provided. In this paper we will focus on the so-called 
\textit{type-II transform}~\cite{dutt1993fast} for which the frequencies $w_k = k\!-\!1$ are equispaced.
The forward problem \cref{eq:Vandermonde_sys}
then corresponds to evaluating a given $1$-periodic Fourier series at
real targets \smash{$\{p_j\}_{j=1}^m$},
which naively involves $\mathcal{O}(mn)$ operations.
There exist well-established algorithms based upon the fast Fourier transform (FFT) \cite{dutt1993fast,barnett2019parallel,kunis2008time,potts2003fast,greengard2004accelerating,ruiz2018nonuniform}
that compute the type-II transform approximately at a reduced complexity $\mathcal{O}(n\log n + m \log(1/\epsilon))$, where $\epsilon$ is an accuracy parameter.
Mature software implementations of these fast transforms are available such as NFFT\footnote{Available at \url{https://github.com/NFFT/nfft}} \cite{keiner2009using} and FINUFFT\footnote{Available at \url{https://finufft.readthedocs.io}} \cite{barnett2019parallel}.
In this article, we focus on the \emph{inverse} problem:
\boxedtext{\textbf{Inverse type-II NUDFT problem.} Given measurements \hbox{$b = (b_1,\ldots,b_m)^T \in \complex^m$} and locations $\{p_j\}_{j=1}^m$,
  determine the coefficients
  $x = (x_1,\ldots,x_n)^T \in \complex^n$. } 
Unlike for the plain discrete Fourier transform (DFT), inversion of the type-II
is not simply an application of its adjoint (type-I) transform; rather, a large
dense $m\times n$ linear
system must be solved.
This task, and its higher-dimensional generalizations,
is important in diverse applications including
geophysics \cite{sacchi1996estimation,zwartjes2007fourier},
astrophysics \cite{swan1982discrete},
signal and image processing \cite{feichtinger95,bagchi2001nonuniform,viswanathan10},
and other areas of computational mathematics
\cite{martin2013computing, pan2001structured}.
In particular, type-II inverse NUDFT solvers can be used to reconstruct an image from irregularly
sampled Fourier data, as occurs in
non-Cartesian magnetic resonance imaging (MRI)
\cite{nattererbook,fessler2005toeplitz,greengard2006fast,kircheis2023direct}
and synthetic aperture radar (SAR) \cite{greengelb23}.

In this article, we develop a new direct solver for the type-II inverse NUDFT problem in the overdetermined setting. The complexity of our method is \smash{$\mathcal{O}((m+n)\log^2 n \log^2 (1/\epsilon))$} in most cases.
Our strategies are related to ideas that have been proposed for solving Toeplitz linear systems~\cite{martinsson2005fast, xia2012superfast,xi2014superfast}, but they appear to be new in this context. The method we develop leverages displacement structure~\cite{kailath1995displacement, koev1999matrices}, low rank approximation, and hierarchical numerical linear algebra~\cite{benzi2016matrices,martinsson2019fast}. It is markedly different from other methods for the inverse NUDFT problem such as iterative solvers \cite{feichtinger95,fessler2005toeplitz,ruiz2018nonuniform,greengard2006fast}
and recent direct methods combining sparse weights with the adjoint transform \cite{kircheis2019direct,kircheis2023direct}.
Our solver is general-purpose and fully adaptive, requiring no tuning parameters from the user other than a prescribed error tolerance. Unlike all of the above-mentioned methods, our approach is highly robust to sampling locations and can be applied in problems where sample locations are clustered or irregular. It is also particularly fast for problems that involve multiple right-hand sides. 
\subsection{Problem formulation}
\label{sec:problem-formulation}

This paper focuses on the inverse type-II NUDFT problem in the one-dimensional setting \cref{eq:Vandermonde_sys}, where the system is overdetermined (i.e., the number of measurements $m$ equals or exceeds the number of unknowns $n$). The extension to type-I NUDFT inversion
is straightforward, and we briefly discuss other extensions in \cref{sec:extensions}.  
Define the NUDFT matrix \smash{$V \in \complex^{m\times n}$} to have entries
\begin{equation*}
  V_{jk} = \gamma_j^{k-1}, \quad \text{where} \quad \gamma_j = e^{-2\pi i p_{j}}, \quad \text{for } j =1,\ldots,m, \, k = 1,\ldots,n.
\end{equation*}
The matrix $V$ is a \emph{Vandermonde matrix}, and $\gamma_1,\ldots,\gamma_m$
are referred to as the \emph{nodes} of $V$.
We assume the sample locations are distinct and ordered so that \smash{$0 \le p_m < \cdots < p_1 <1$}. This ensures that the nodes are unique and arranged counterclockwise around the unit circle. One can permute the rows of $V$ to order the sample locations in this way. 
The inverse type-II NUDFT problem consists of solving $Vx = b$ in the least-squares sense:
\begin{equation} \label{eq:least_squares}
  x = \argmin_{x \in \complex^n} \norm{Vx - b}_2^2.
\end{equation}
We include the noisy-data case where $b$ is not in the range of $V$,
important in applications where correct averaging over an excess of data
reduces the noise in the solution $x$. 
Since $\gamma_j$ are distinct, $V$ has full rank. While we allow that $V$ may be ill-conditioned, we assume that $V$ is numerically full-rank, in the sense that the condition number
\begin{equation}
  \label{eq:condition}
  \kappa_2(V) = \frac{\sigma_{\rm max}(V)}{\sigma_{\rm min}(V)}
\end{equation}
is much smaller than the inverse machine precision (i.e., $\kappa_2(V) \le 10^{14}$ for nearly compatible problems in double precision arithmetic).
Here, $\sigma_{\rm max}(V)$ and $\sigma_{\rm min}(V)$ denote the largest and smallest singular values of $V$.

\begin{rmk}\label{r:conventions}   
  There are many conventions in the literature for the
  type-II (forward) transform, varying in the sign of the exponential,
  presence or absence of the $2\pi$ factor, and 
  whether frequency indices are symmetric about zero
  (compare, e.g., \cite{dutt1993fast,keiner2009using,barnett2019parallel,kircheis2019direct,matlabnufft}).
  Ours
  is close to that for the plain DFT.
  Conversions between conventions can be performed numerically using
  simple linear-time operations on inputs or outputs.
\end{rmk}

\subsection{From Vandermonde to Cauchy-like}
The fundamental idea behind our direct inversion method is to transform from a problem involving a Vandermonde matrix to an equivalent problem involving a Cauchy-like matrix \smash{$C = VF^*$}, where $F$ is a DFT matrix defined below in~\cref{eq:F}. This transformation is useful because  $C$ has low rank properties that can be analyzed and exploited~\cite{heinig1995inversion}. Similar ideas have been proposed for superfast direct Toeplitz solvers~\cite{heinig1995inversion, wilber2021computing,xia2012superfast,xi2016computing}, though key details turn out to be quite different in this setting.

We begin by observing that $V$ satisfies the Sylvester matrix equation
\begin{equation} \label{eq:vand_disp}
\Gamma V - VQ = uv^*, 
\end{equation}
where \smash{$(\cdot)^*$} denotes the Hermitian transpose, \smash{$uv^*$} is a rank 1 matrix, $Q =  \left( \begin{smallmatrix} 0 & 1\\ I & 0  \end{smallmatrix} \right) \in \mathbb{C}^{n \times n}$ is the circular shift-down matrix, and
\begin{equation*}
  \Gamma = \diag(\gamma_1, \ldots, \gamma_m).
\end{equation*}
Specifically, $u$ has $j$th entry \smash{$\gamma_j^n-1$}, while \smash{$v^* = [0,\dots,0,1]$}.
Let $F$ be the normalized DFT matrix
\begin{equation}
\label{eq:F}
  F_{jk} = \omega^{j(2k-1)}/\sqrt{n} \quad \text{for } 1\le j,k \le n \quad \text{where}\quad \omega = e^{ \pi i/n}.
\end{equation}
The circulant matrix $Q$ is diagonalized by $F$:
\begin{equation*}
  F Q F^* = \Lambda \coloneqq \diag(\omega^{2}, \omega^{4}, \ldots, \omega^{2n}).
\end{equation*}
Letting \smash{$C = VF^*$}, it follows from~\cref{eq:vand_disp} and \smash{$F F^*=I$} that 
\begin{equation} \label{eq:diag_sylv}
  \Gamma C - C\Lambda = uw^*, 
\end{equation}
where we have defined $w = Fv$,          
and this in turn implies that
\begin{equation} \label{eq:cauchystrucv}
  C_{jk} = \frac{u_j\overline{w_k}}{ \Gamma_{jj} - \Lambda_{kk}} = \frac{u_j \overline{w_k}}{\gamma_j - \omega^{2k}}.
\end{equation}
A matrix with the structure\smash{ $C_{jk} = a_jb_k / (c_j - d_k)$} for numbers \smash{$\{a_j\}$, $\{b_k\}$, $\{c_j\}$}, and \smash{$\{d_k\}$} is called a \emph{Cauchy-like matrix}~\cite{koev1999matrices}. The algebraic structure of these matrices is useful in various ways~\cite{kailath1995displacement}.  For example, one can generate subblocks of $C$ using~\cref{eq:cauchystrucv} (or related formulas), which only require accessing entries in four vectors. In addition to its algebraic structure, $C$ possesses many submatrices that have low numerical rank.  As we detail precisely in \cref{sec:rank_structs}, the low rank structure of $C$ depends on the distribution of the nodes \smash{$\{\gamma_1,\ldots,\gamma_m\}$} relative to
the $n$th roots of unity
\hbox{\smash{$\{\omega^{2}, \omega^{4}, \ldots, \omega^{2n}\}$}}.
Our analysis of the rank structure of $C$ is constructive and uses the displacement relation \cref{eq:diag_sylv}. We show that one can cheaply generate low rank approximations for these submatrices using small Sylvester matrix equations related to~\cref{eq:diag_sylv}. This allows us to construct an hierarchical approximation to $C$ while never explicitly forming or storing $C$ and its relevant submatrices.

\subsection{A displacement-based superfast solver}

Our proposed direct solver approximates a least-squares solution to
$Cy=b$, and then returns \smash{$x = F^*y$}.
The latter is an approximate solution to \cref{eq:least_squares},
recalling that \smash{$C=VF^*$}.
The solver is described in pseudocode in \cref{alg:generalsolver}.
It employs the following ingredients, which we develop over the course of the rest of the paper:
\begin{itemize}
\item \textbf{Hierarchical rank structure.} In \cref{sec:low-rank-properties}, we show that $C$ can be approximated to high accuracy by a matrix $H$ which possesses hierarchical low rank structure under the hierarchical semiseparable (HSS) format~\cite{chandrasekaran2006fast, chandrasekaran2007superfast,martinsson2011fast}.
  For an approximation tolerance $\epsilon$, the off-diagonal blocks of $H$ have rank at most $k$, where 
  \begin{equation} \label{eq:rank-bound}
    k \le \left\lceil \frac{2\log(4/\epsilon)\log(4n)}{\pi^2} \right\rceil = \order(\log (n) \log(1/\epsilon)).
  \end{equation}
\item \textbf{Superfast construction by ADI.} We provide a construction algorithm for $H\approx C$ that combines the factored alternating direction implicit (ADI) method~\cite{benner2009adi,townsend2018singular} with an interpolative decomposition~\cite{cheng2005compression}. This algorithm runs in complexity \newline  \smash{$\order(mk^2) = \order(m \log^2 (n) \log^2(1/\epsilon))$} time.
\item \textbf{Superfast least-squares solver.} For a square HSS matrix $H$ with off-diagonal blocks of rank $k$ , the linear system $Hx = b$ can be solved in \smash{$\order(mk^2)$} operations~\cite{chandrasekaran2006fast, martinsson2011fast}. Less work has focused on the least-squares case. A semidirect method based on recursive skeletonization is described in~\cite{ho2014fast}. The  solver we develop here is a generalization of the specialized least-squares solver for overdetermined Toeplitz systems introduced in~\cite{xi2014superfast}.  
\end{itemize}
While our compression strategy takes special advantage of the displacement structure of $C$, our least-squares solver (step 4 in \Cref{alg:generalsolver}) is generic and can be applied to any rectangular system with HSS rank structure. 
 
 \begin{algorithm}[t!]
\caption{A superfast least-squares solver for $Vx = b$.  (Type-II NUDFT inversion)}
\label{alg:generalsolver}
\begin{algorithmic}[1]
  \State Compute \smash{$w = Fv$}, where $v$ is as in~\cref{eq:vand_disp}.

\State Use ADI and~\cref{eq:diag_sylv} to generate \smash{$H \approx C$}, where $H$ is a rectangular HSS matrix (see \cref{sec:rank_structs})
\State Solve \smash{$H y = b$} in the least-squares sense (see \cref{sec:hier-semis-least}). 
\State Compute \smash{$x = F^* y.$} 
\end{algorithmic}
\end{algorithm}

\subsection{Related work} 
\label{sec:comp-prev-work}

There are many approaches to the inverse NUDFT problem,
and we can only give a brief overview (also see \cite[Sec.~1]{kircheis2023direct}).
Standard approaches include (i) working directly with the system $Vx = b$,
(ii) working with the normal equations, \smash{$V^*Vx = V^*b$}, and (iii) working with the the adjoint
(or ``second kind'')
normal equations, \smash{$VV^*y = b$}.
Variants which sandwich a diagonal weight matrix between $V$ and \smash{$V^*$} also exist
\cite{feichtinger95,kircheis2023direct}.

A major class of methods for solving these systems is the iterative methods,
which can leverage fast implementations of the forward NUDFT transform.
Since the normal and adjoint normal linear systems are
positive semidefinite, the conjugate gradient (CG) method can be applied~\cite{saadbook,feichtinger95,fessler2005toeplitz,ruiz2018nonuniform}. This is especially appealing
for the normal equations since
\smash{$V^*V$} is a Toeplitz matrix, for which a fast matrix-vector product is available via a
pair of padded FFTs (whose prefactor is smaller than that of NUDFT transforms).
The adjoint normal matrix \smash{$VV^*$} is not in general Toeplitz, but may still be applied
via a pair of NUDFTs (a type-I followed by a type-II).
Iterative methods are quite effective as long as $V$ is well-conditioned.  However, convergence rates for these methods depend on the condition number \smash{$\kappa_2(V)$}, defined in \cref{eq:condition}. Unless strong restrictions are placed on the sample locations~\cite{donoho89,feichtinger95,ruiz2018nonuniform,yu2023stability}, \smash{$\kappa_2(V)$} can increase without bound.
Diagonal preconditioners for the above normal systems have been developed,
based on local sampling spacings (Voronoi weights) \cite{feichtinger95},
or weights which optimize the Frobenius-norm difference from the
pseudoinverse (so-called ``\smash{sinc$^2$} weights'') \cite{greengard2006fast}.
These are related to density compensation (or ``gridding'') weights in MRI
\cite{nattererbook}.
Moving beyond diagonal preconditioners,
Toeplitz systems also have various circulant preconditioners
available that exploit the FFT \cite{chan2007introduction}.

Another class of methods are the direct methods. These have some advantages over iterative methods in some settings:
\begin{itemize}
\item \textbf{Poor conditioning.}
  In many applications, the sample locations may be highly irregular, clustered,
  or have regions of low density relative to the shortest Fourier wavelength.
  For these problems, $V$ is ill-conditioned,
  and iterative methods may converge slowly or not at all
  \cite{feichtinger95}.
  In contrast, the runtime of a direct solver does not depend on problem conditioning. 
\item \textbf{Multiple right-hand sides.} Some applications involve solving a sequence of inverse NUDFT problems with the same sample locations \smash{$\{p_j\}_{j = 1}^m$} but many different right-hand sides $b$. For direct solvers which factor the matrix $V$ (or the transformed matrix $C$), the factorization can be stored and reused to solve multiple right-hand sides quickly (e.g, for our solver, $\mathcal{O}(mk)$ operations per right-hand side for the numerical rank $k$ considered in \cref{eq:rank-bound}). 
\end{itemize}

The \smash{${\mathcal O}(mn^2)$} cost
of dense, unstructured direct methods (e.g., column-pivoted QR) makes them essentially useless beyond
a few tens of thousands of unknowns.
Turning to fast direct methods, one class is based on the normal
equations, which can be solved by fast~\cite{kailath1995displacement,gohberg1995fast} or superfast direct Toeplitz solvers~\cite{xia2012superfast,wilber2021computing}. However, for ill-conditioned problems, normal equations-based direct solvers can have substantially reduced accuracy as the condition number for the normal equations is equal to the square of the original matrix, \smash{$\kappa_2(V^*V) = \kappa_2(V)^2$}. Our direct solver works directly with $V$ to avoid this issue. 

Another direct method was developed in~\cite{kircheis2019direct}, which we refer to as the Kircheis--Potts algorithm. This method is based on an optimization routine that finds a specially structured matrix $\Psi$ such that $V\Psi  \approx mI$. Constructing $\Psi$ requires \smash{$\mathcal{O}(m^2 +n^2)$} operations, which can be prohibitive in large-scale settings. However, once constructed, $\Psi$ can be applied to right-hand sides to solve the inverse NUDFT problem at the cost of an FFT plus some sparse matrix-vector products. We find (see \cref{sec:Direct})
that this method works well when the nodes of $V$ are clustered without irregular gaps.   

Recently Kircheis and Potts proposed a direct solution via
\smash{$x \approx V^* \! \diag(\mu) b$}, where the weight vector $\mu$ is precomputed by solving  a

adjoint Vandermonde system of larger size $(2n\!-\!1)\times m$ \cite[Sec.~3]{kircheis2023direct}.
The latter enforces that the weights \smash{$\mu_j$} at nodes \smash{$p_j$} form an exact quadrature
scheme on $p\in(0,1)$ for all modes \smash{$e^{2\pi i k p}$} with $k=-n\!+\!1,\dots,n\!-\!1$,
which can only in general be consistent when $m\ge2n\!-\!1$.
However, to achieve a fast scheme this system must be solved iteratively (again by CG on normal or adjoint normal
equations),
and in our experience this converges no more reliably than
the original NUDFT system.

The direct method for $Vx =b$ that is perhaps closest in spirit to our own work is that of~\cite{dutt1995fast}: this evaluates the Lagrange interpolant (cotangent kernel)
using a fast multipole method, and
is unfortunately limited to the setting where $V$ is square.  There is a long history of methods that have been proposed for solving square systems involving Cauchy-like and related matrices, including $\mathcal{O}(n^2)$ methods based on exploiting the preservation of displacement structures in Schur complements~\cite{gohberg1995fast}. Heinig supplies a brilliant precursor to superfast methods for Toeplitz, Hankel, and Cauchy-like matrices~\cite{heinig1995inversion}, including a divide-and-conquer $\mathcal{O}(n \log^3 n)$ method for square systems involving well-conditioned Cauchy-like matrices.

\begin{rmk}[Overdetermined vs underdetermined problems]\label{r:types}

  As the above (and \cref{sec:experiments}) suggests,
  the condition number \smash{$\kappa_2(V)$} of a problem has a dramatic effect on the relative merits of different algorithms
  (iterative vs direct).
  It is thus worth distinguishing
  three classes of analysis results on \smash{$\kappa_2(V)$} in the literature.
  \begin{itemize}
  \item
  When the type-II linear system $Vx=b$ is overdetermined,
  as in our study, \smash{$\kappa_2(V)$} controls the sensitivity of the solution
  \cite[Ch.~15]{nla}
  (loosely, the factor by which $\|x\|$ may exceed $\|b\|$).
  
  The absence of {\em gaps} in the sampling set \smash{$\{p_j\}_{j=1}^m$}
  appears to guarantee that \smash{$\kappa_2(V)$} remain small
  (see \cite{donoho89}, and for Voronoi weighting see \cite[Prop.~2]{feichtinger95} and \cite[Lem.~4.2]{adcock14}).
  Conversely, gaps somewhat larger than the wavelength $\pi/n$ at the Nyquist
  frequency $n/2$ appear to guarantee ill-conditioning. 
  \item
  For the square case $m=n$, the problem becomes one of interpolation
  rather than least-squares fitting
  \cite{dutt1995fast}, and \smash{$\kappa_2(V)$} still reflects its sensitivity.
  Recent results show that this is well conditioned when the sample points
  have deviations from uniformity limited to a fraction of the uniform
  spacing $1/n$
  \cite{austin2017trigonometric,yu2023stability}.
  
\item
    For the underdetermined type-II case $m<n$, $Vx=b$ always has a nullspace, so
    regularization is always needed \cite{kunis2007stability} or one must choose an appropriate minimum-norm solution.
      However, \smash{$\kappa_2(V)$}
      instead describes the sensitivity of the adjoint (type-I) overdetermined linear system \smash{$V^*y = b$},
      interpreted as solving for the $n$ {\em strengths} \smash{$y_k$} whose
      type-I transform matches a given set of Fourier coefficients \smash{$b_j$}, $j=1,\dots, m$.
      This problem is well studied in the {\em superresolution} literature,
      where analysis shows polynomial
      ill-conditioning associated with clustering of
      (as opposed to gaps in) the nodes
      \cite{bazan2000conditioning,batenkov2019spectral,kunis21}. 
      Here \smash{$\kappa_2$} is controlled by the {\em minimum}
      (as opposed to maximum) node separation,
      but this class of results has no bearing on the sensitivity of
      \cref{eq:least_squares}.
      Indeed, we will see that clustering does {\em not}
      in itself induce poor conditioning (see Grid 2 below).   %
      \end{itemize}
\end{rmk}
    
\subsection{Organization} The rest of the paper is organized as follows: In \cref{sec:low-rank-properties}, we briefly review the connection between low rank approximation and displacement rank. Then in \cref{sec:rank_structs}, we use this framework to explain low rank properties of the transformed matrix $C$ and develop a construction algorithm for an HSS matrix that approximates $C$ well. In \cref{sec:hier-semis-least}, we describe the rectangular HSS solver used in our direct inversion scheme. Numerical experiments are discussed in \cref{sec:experiments},
starting with a summary of recommendations for which solver to use in
which regime. A brief discussion of extensions and future work is given in \cref{sec:extensions}. The five iterative methods against which we
compare are detailed in an appendix.

\section{The singular values of matrices with low displacement rank}
\label{sec:low-rank-properties}
Our solver makes use of \emph{displacement structure}, a property possessed by many highly structured matrices, including Toeplitz, Cauchy, Hankel, and Vandermonde matrices \cite{beckermann2016singular}. A matrix $X$ is said to have $(A, B)$-displacement structure if it satisfies the Sylvester matrix equation
\begin{equation} \label{eq:displacement}
  AX - XB = M
\end{equation}
for some low rank matrix $M$. The rank of $M$ is called the displacement rank of $X$. As seen in~\cref{eq:diag_sylv},  $C$ has $(\Gamma, \Lambda)$-displacement structure with a \emph{displacement rank} of $1$.  Following~\cite{beckermann2016singular}, we use the displacement structure of $C$ to explain its low rank properties. We begin with a brief review of the connection between displacement structure and singular value decay.

Consider the displacement relation \cref{eq:displacement}, where $A$ and $B$ are normal matrices. 
We assume the spectra of $A$ and $B$ are contained by bounded, disjoint sets \smash{$E_1$} and \smash{$E_2$} in the complex plane: \smash{$\lambda(A) \subset E_1$ and $\lambda(B) \subset E_2$}. It is shown in~\cite{beckermann2016singular} that
\begin{equation} 
\label{eq:zolobound_svs}
\sigma_{k \rho + 1}(X) \leq Z_{k}(E_1, E_2) \|X\|_2,
\end{equation}
where $\rank(M) \leq \rho$, and \smash{$Z_{k}(E_1, E_2)$} is the $k$th \textit{Zolotarev number} associated with \smash{$E_1, E_2$}, defined to be the solution to a certain a rational extremal problem involving \smash{$E_1$} and \smash{$E_2$}. For a detailed overview of Zolotarev numbers in this context, see~\cite{beckermann2016singular, wilber2021computing}. When \smash{$E_1$} and \smash{$E_2$} are well-separated, \smash{$Z_k$} decays rapidly and we therefore expect that $X$ is of low numerical rank. For our purposes, \smash{$E_1$} and \smash{$E_2$} will always be selected as arcs on the unit circle. The following lemma, restated from~\cite{wilber2021computing}, is crucial for our results:

\begin{lemma} \label{lemma:arcs}
Let $\mathcal{A}_J = \{ e^{it} : t \in [\tau_1, \tau_2]\}$, $\mathcal{A}_K = \{ e^{it} : t \in [\rho_1, \rho_2]\}$. Then, 
\begin{equation} \label{eq:bounds_arcs}
Z_k(\mathcal{A}_J, \mathcal{A}_K) \leq 4 \mu^{-2k}, \quad \mu = \exp\left( \frac{\pi^2}{2 \log (16 \eta) } \right), 
\end{equation}
where 
\begin{equation}
 \eta = \frac{| \sin( (\rho_1- \tau_1)/2 ) \sin ( (\rho_2 - \tau_2)/2 )|}{ | \sin( (\rho_2 - \tau_1)/2) \sin ( (\rho_1 - \tau_2) /2) | }.
 \end{equation}
\end{lemma}

\begin{proof} See Theorem 8 in~\cite[Ch.~4]{wilber2021computing}.
\end{proof}

Bounds on the the singular values of $X$ imply estimates for its so-called $\epsilon$-rank, which is defined as follows:
\begin{definition}
Let $0 < \epsilon < 1$ and $m \geq n$. The $\epsilon$-rank of \smash{$X \in \mathbb{C}^{m \times n}$}, denoted by \smash{$\erank(X)$}, is the smallest integer $0 \leq k \leq n\!-\!1$ such that \smash{$\sigma_{k+1}(X) \leq \epsilon \|X\|_2$}. 
\end{definition}
By definition, if \smash{$\erank(X) \leq k$}, then there is a rank-$k$ matrix \smash{$X_k$} where \smash{$\|X -X_k\|_2 \leq \epsilon \|X\|_2$}. 

\subsection{The factored alternating direction implicit method} Intimately related to the Zolotarev numbers is the alternating direction implicit (ADI) method~\cite{lebedev1977zolotarev, peaceman1955numerical} for solving the Sylvester matrix equation \cref{eq:displacement}. Setting \smash{$X^{(0)} = 0$}, each ADI iteration consists of two steps: 
\vspace{.25cm}
\begin{enumerate}
\item  Solve for $X^{(j + 1/2)}$ in $(A-\beta_{j+1}I)X^{(j + 1/2)} = X^{(j)} (B - \beta_{j + 1}I) + M$. 
\item  Solve for $X^{(j + 1)}$ in $X^{(j+1)}(B - \alpha_{j + 1}I) = (A-\alpha_{j+1}I)X^{(j+1/2)}-M$. 
\end{enumerate}
\vspace{.25cm}
After $k$ iterations, an approximation \smash{$X^{(k)} \approx X$} is constructed. 
Convergence of the method depends on the choice of \textit{shift parameters} \smash{$\{\alpha_j, \beta_j\}_{j = 1}^k$}. The shift parameters that minimize the error after $k$ iterations are given by the zeros and poles of a type $(k,k)$ rational function that solves an optimization problem on the spectral sets $\lambda(A)$ and $\lambda(B)$ known as Zolotarev's third problem~\cite{istace1995third,lebedev1977zolotarev, zolotarev1877application}. 
When optimal shift parameters are used and we let $E_1, E_2$ serve as stand-ins for the spectral sets that they contain, one can show that~\cite{townsend2018singular}
\begin{equation*}
  \big\|X - X^{(k)}  \big\|_2 \leq Z_k(E_1, E_2) \|X\|_2 .
\end{equation*}
In our case, $E_1$, $E_2$ are always arcs on the unit circle. The optimal shift parameters for this case are known and can be computed at a trivial cost~\cite[Ch.~1]{wilber2021computing} using elliptic functions.

A mathematically equivalent version of ADI called factored ADI (fADI)~\cite{benner2009adi, li2002low} constructs the solution \smash{$X^{(k)}$} in the form of a low rank decomposition \smash{$X^{(k)} = ZW^*$}, where $Z, W$ each have $k\rho$ columns. It uses as input a factorization \smash{$M = FG^*$}, where \smash{$F \in \mathbb{C}^{m \times \rho}$}, \smash{$G \in \mathbb{C}^{n \times \rho}$}.  The fADI algorithm is our primary engine for constructing low rank approximations to submatrices of $C$.  The cost of this method depends on the cost for shifted matrix-vector products $z\mapsto (M+\gamma I)z$ and shifted inverts $z\mapsto (M+\gamma I)^{-1}z$ for $M\in \{A,B\}$.
In our case, $A$ and $B$ will always be diagonal, so the cost for applying $k$ iterations of fADI on a $p \times q$ submatrix of $C$ is only $\mathcal{O}((p+q)k)$ operations. We supply pseudocode for this procedure in \cref{sec:fadi_alg}. Detailed  discussion on the connection between Zolotarev numbers and ADI can be found in~\cite{townsend2018singular, wilber2021computing}.

\section{The low rank properties of $C$} \label{sec:rank_structs}
As shown in~\cref{eq:diag_sylv},  the matrix $C$ has $(\Gamma, \Lambda)$ displacement structure with displacement rank $1$.  Since $\Gamma$ and $\Lambda$ potentially have interlaced (and possibly coinciding) eigenvalues,~\cref{eq:diag_sylv} is typically not useful for understanding the singular values of the entire matrix $C$.  However,~\cref{eq:diag_sylv} can be used to show that certain \emph{submatrices} of $C$ have low numerical rank.
Throughout, we use the notation $C_{JK}$ to denote the submatrix of $C$ with rows indexed by members of set $J$ and columns indexed by members of set $K$.

We begin with an example.
Suppose the index sets are \hbox{$J = \{q \! + \! 1,q \! + \! 2, \ldots, m\}$} and \hbox{$K = \{1, 2, \ldots, p\}$}, so that \smash{$C_{JK}$} is a submatrix in the lower left corner of $C$. By~\cref{eq:diag_sylv}, 
\begin{equation}
\label{eq:dispY}
 \Gamma_J C_{JK} - C_{JK} \Lambda_K = (uv^*)_{JK},
\end{equation} 
where \smash{$\Gamma_J = \diag( \gamma_{q+1}, \ldots, \gamma_{m})$} and \smash{$\Lambda_K = \diag( \omega^2, \ldots, \omega^{2p} )$}.
Choose arcs \smash{$\mathcal{A}_J$} and \smash{$\mathcal{A}_K$} so that \smash{$\lambda(\Gamma_J) \subset \mathcal{A}_J$,}  \smash{$\lambda(\Lambda_K) \subset \mathcal{A}_K$}. Whenever these two arcs do not overlap, \cref{lemma:arcs} shows that the singular values of $C_{JK}$ decay rapidly. Indeed, by~\cref{eq:zolobound_svs,lemma:arcs},
$$ \sigma_{k+1}(C_{JK}) \leq 4 \mu^{-2k} \|C_{JK}\|_2,$$
which implies bounds on \smash{$\erank(C_{JK})$}.  Moreover, we can use fADI to efficiently construct a low rank approximation to \smash{$C_{JK}$} that achieves these estimates. This gives us a scheme for compressing submatrices of $C$ that are of low numerical rank.
In the rest of this section, we use these methods to describe the hierarchical rank structure of $C$ and construct a rank-structured approximation $H\approx C$.

\subsection{A geometric perspective}

\begin{figure}[t!] 
  \centering
  \begin{minipage}{.42\textwidth} 
    \centering
    \begin{overpic}[width=\textwidth]{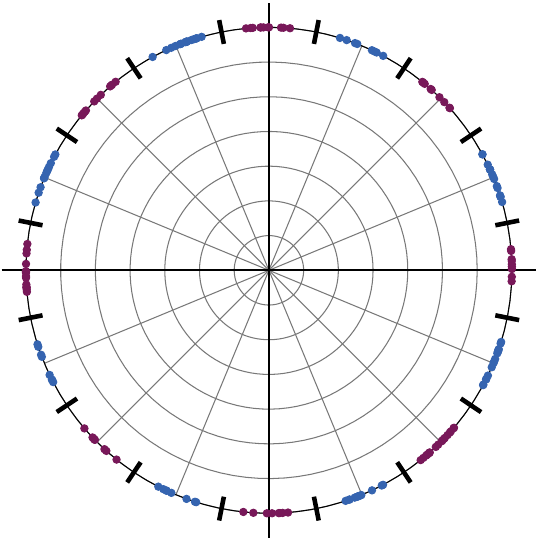}
      \put(45,-5) {\small{$Im$}}
      \put(100, 48) {\small{$Re$}}
      \put(90,75){\rotatebox{-69}{$\overbrace{\phantom{xxxx}}^{\phantom{x}\cluster_1}$}}
      \put(77,87){\rotatebox{-45}{$\overbrace{\phantom{xxxx}}^{\phantom{x}\cluster_2}$}}
      \put(61,95){\rotatebox{-25}{$\overbrace{\phantom{xxxx}}^{\phantom{x}\cluster_3}$}}
    \end{overpic}
  \end{minipage}
  \hspace{1cm}
  \begin{minipage}{.38\textwidth} 
    \hspace{1cm}
    \begin{overpic}[width=.50\textwidth]{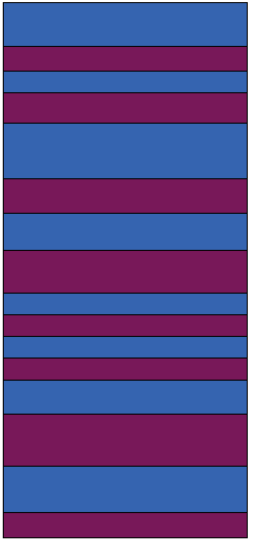}
    \end{overpic}
    \vspace{-.3cm}
    \put(-85,141){\tiny{$\cluster_1$}}
    \put(-85,133){\tiny{$\cluster_2$}}
    \put(-85,126){\tiny{$\cluster_3$}}
    \put(-80,110){\small{$\vdots$}}
    \put(-85,2){\tiny{$\cluster_{16}$}}
  \end{minipage}
  \caption{\emph{Left}: The unit circle is split into $n = 16$ regions. At the center of each region is the root of unity \smash{$\omega^{2\kappa}$}, and each region is associated with a subset of nodes (blue and magenta points) that are indexed by the cluster \smash{$\cluster_\kappa$}. Each cluster indexes those nodes \smash{$\gamma_j$} between \smash{$\omega^{2\kappa \pm 1}$}. 
    \emph{Right}: A cartoon portrayal of the $m \times n$ matrix $C$ associated with these nodes, partitioned into $n$ slabs. The number of rows in each slab corresponds to the size of the associated cluster. Our HSS approximation to $C$ works with these slabs instead of individual rows.}
  \label{fig:clusters}
\end{figure}
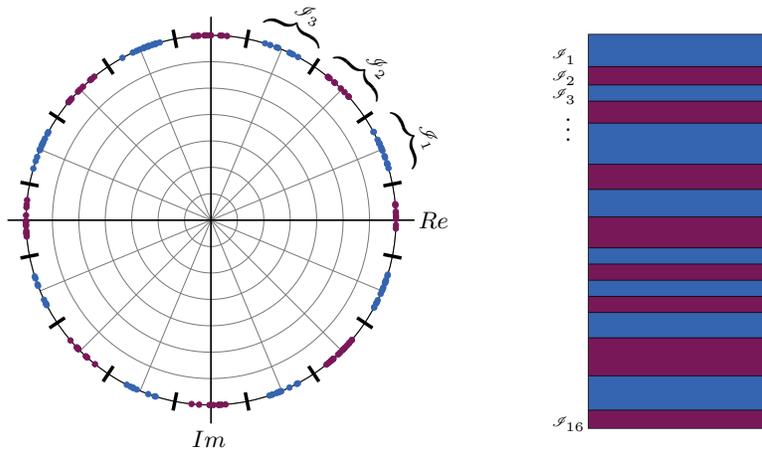

Not every submatrix of $C$ is compressible. However,  $C$ always contains some collection of submatrices with small $\epsilon$-rank. The locations and dimensions of these submatrices depends on the distribution of the nodes \smash{$\{\gamma_j\}_{j = 1}^m$} of the Vandermonde matrix $V$ on the unit circle. In an extreme example where all of the nodes are clustered between two adjacent roots of unity, $C$ is a low rank matrix. Similarly, one might have the case where a large fraction of the nodes are tightly clustered near only a very small collection of roots of unity. We address this more pathological setting in \cref{rem:pathological}. 

For less extreme cases, we seek a general strategy for recursively partitioning $C$ into a collection of compressible submatrices. To do this, we organize the partitioning around subsets of nodes that we index using \textit{clusters}.
\begin{definition}  For $1\le \kappa \le n$, the \emph{cluster} \smash{$\cluster_\kappa$} consists of all $j$ such that
$    \gamma_j \in \{ \exp(2\pi i t/n) : \kappa - 1/2 < t \le \kappa+ 1/2 \}.$
\end{definition}
Each cluster $\cluster_\kappa$ gives rise to a slab of rows of $C$ (see \Cref{fig:clusters}).

The use of slabs allows us to essentially treat $C$ as if it were square for the purposes of partitioning, as there are $n$ columns of $C$ and $n$ slabs of rows. It is possible for some clusters to be empty. In the next section, we apply standard hierarchical partitioning strategies to $C$ with slabs replacing rows. Then, in \cref{sec:rank-structure-c}, we bound the $\epsilon$-ranks of relevant submatrices that arise from these partitions. 

\begin{rmk} \label{rem:pathological}
In extreme cases, the number of rows in some slabs  might be a large fraction of $m$. This happens when many nodes are clustered around only a few roots of unity. If this is true and there is also a subset of nodes more regularly distributed around the unit circle, one can write \smash{$C = \tilde{C} + XY^T$}, where every column in $X$ corresponds with a root of unity around which many nodes are clustered, and \smash{$\tilde{C}$} is Cauchy-like and full rank. If $C$ were square, the system $Cy = b$ could then be solved efficiently using the Sherman-Morrison Woodbury (SMW) formula. One can in principle apply instead an adapted version of the SMW formula for least--squares problems~\cite{guttel2024sherman}, though this requires either a fast minimum-norm solver (see \cref{sec:extensions}) for \smash{$\tilde{C}^*y = b$} or an efficient formulation for the inversion of \smash{$\tilde{C}^*\tilde{C}$}.  Other pathological clustering configurations should be handled on a case-by-case basis: \cref{lemma:arcs} can be applied to identify low rank structures in the matrix, but the best partitioning strategy and subsequent solver might be quite different from what we have proposed.
\end{rmk}
 
\subsection{HODLR matrices} \label{sec:cHODLR} With a slight abuse of terminology, we say that $C$ is a \emph{HODLR (hierarchical off-diagonal low rank) matrix} 
if it can be divided into blocks
\begin{equation} \label{eq:HODLR_def}
C = \begin{bmatrix} D_{\ell} & A_{\ell r} \\ A_{r \ell} & D_{r}   \end{bmatrix},
\end{equation}
where (i) the off-diagonal blocks \smash{$A_{\ell r}$} and \smash{$A_{r \ell}$} are well-approximated by low rank matrices and  (ii) the diagonal blocks \smash{$D_{\ell}$} and \smash{$D_{r}$} can be recursively partitioned to have the same form as~\cref{eq:HODLR_def}. The partitioning is repeated recursively until a minimum block size is reached. For us, the matrices \smash{$D_{\ell}$} and \smash{$D_{r}$} are rectangular, and the number of rows in each depends on how the nodes are clustered.

Let us be more precise about this structure. We associate the HODLR matrix with a binary tree $\tree$ whose vertices we label by a post-order traversal: We label the root as $0$ and the children of $t \in \tree$ by $2t\!+\!1$ and $2t\!+\!2$ (see \Cref{fig:HODLRandTree}).
To each vertex $t \in \tree$, we associate a set of column indices \smash{$K_{t}$} such that (i) $\smash{K_0 = \{1,\ldots,n\}}$, (ii) each block of indices is contiguous, i.e.,  \smash{$K_t = \{\kappa_1,\kappa_1\! +\! 1,\ldots,\kappa_2\}$} for some $1\le \kappa_1 \le \kappa_2 \le n$, and (iii) each parent is the disjoint union of its children, i.e., \smash{$K_t = K_{2t+1} \cup K_{2t+2}$}. We then define \smash{$J_t$} to be a union of the clusters \smash{$\kappa \in K_t$}, i.e., \smash{$J_t \coloneqq \bigcup_{\kappa\in K_t} \cluster_\kappa$}.
We refer to \smash{$D_t \coloneqq C_{J_t K_t}$} as a \emph{diagonal} HODLR block. The blocks \smash{$A_{r\ell} \coloneqq C_{J_rK_\ell}$} and \smash{$A_{\ell r} \coloneqq C_{J_\ell K_r}$} are called \emph{off-diagonal} HODLR blocks.
An example with multiple levels is shown in \cref{fig:HODLRandTree}. As we prove in \cref{sec:rank-structure-c}, the off-diagonal blocks $A_{\ell r}$ and $A_{r\ell}$ have low (numerical) rank at every level of the tree.

In the special case where $m = n$ and each \smash{$\cluster_\kappa$} contains exactly one node,  the cluster-based HODLR structure we use is equivalent to the usual HODLR structure for square matrices associated with perfectly balanced binary trees~\cite{chandrasekaran2006fast, martinsson2011fast}. This corresponds to
sample points \smash{$p_j$} with bounded jitter relative to the equispaced case,
that is \smash{\hbox{$|p_{m} - (m\!-\!j)/m|  < 1/2m$}} for $j=1,\dots,m$.
However, in general, our paradigm allows for rectangular matrices with uneven divisions of rows at any  fixed level of the tree (see \Cref{fig:HODLRandTree}). 

 One could use fADI to construct low rank approximations to every relevant off-diagonal block in $C$, and then apply a solver that exploits HODLR structure to approximately solve \hbox{$Cy = b$}.  However, we do not fully explore this possibility because, in addition to HODLR structure, $C$ possesses further structure can be exploited for greater efficiency. We build on the HODLR properties of $C$ to describe these structures in the next section.

\begin{figure}
\begin{minipage}{.39\textwidth}
\begin{tikzpicture}[sibling distance=5.5pt,scale=0.9, every node/.style={scale=0.9}]
\tikzset{every tree node/.style={align=center,anchor=north}}
\tikzset{level 1/.style={level distance=35pt}}
\tikzset{level 2/.style={level distance=40pt}}
\tikzset{level 3+/.style={level distance=40pt}}
\hspace{.1 cm}  \Tree [ .\node[]{$0$ };  
										[ .\node[]{$1$ }; 
											[ .\node[]{$3$ };
												[.\node[]{$7$};]
												[.\node[]{$8$};]
											]
											[ .\node[]{$4$ };
												[.\node[]{$9$};]
												[.\node[]{$10$};]
											]
										] 
									    [ .\node[]{$2$ }; 
									    	[ .\node[]{$5$  };
									    		[.\node[]{$11$};]
												[.\node[]{$12$};]
									    	]
											[ .\node[]{$6$};
												[.\node[]{$13$};]
												[.\node[]{$14$};]
											]
										] 
									]

\end{tikzpicture}
\end{minipage}
\begin{minipage}{.32\textwidth}
\begin{tikzpicture}[scale=0.55, every node/.style={scale=0.4}]
\draw[black, thick](0,0)--(6,0);
	\draw[black, thick] (0,0)--(0,8);
\draw[black, thick] (0,8)--(6,8);
\draw[black, thick](6,0)--(6,8);
\draw[black, thick](0,3)--(6,3);
\draw[black, thick](3,8)--(3,0);
\draw[black, thick](1.5,8)--(1.5,3);
\draw[black, thick](0,6)--(3,6);
\draw[black, thick](4.5,0)--(4.5,3);
\draw[black, thick](3,2)--(6,2);  
\node[font=\fontsize{18}{18}] at (4.5,5.5) (a) {$A_{1,2}$};
\node[font=\fontsize{18}{18}] at (1.5,1.5) (a) {$A_{2,1}$};
\node[font=\fontsize{16}{16}] at (2.25,7) (a) {$A_{3,4}$};
\node[font=\fontsize{16}{16}] at (.75,5) (a) {$A_{4,3}$}; 
\node[font=\fontsize{16}{16}] at (5.25,2.5) (a) {$A_{5,6}$};
\node[font=\fontsize{16}{16}] at (3.75,1) (a) {$A_{6,5}$};  
\filldraw [fill=gray!30] (0,6) rectangle (1.5,8);
\filldraw [fill=gray!30] (1.5,3) rectangle (3,6);
\filldraw [fill=gray!30] (3,2) rectangle (4.5,3);
\filldraw [fill=gray!30] (4.5,0) rectangle (6,2);
\node[font=\fontsize{16}{16}] at (.75, 7) (a) {$D_{3}$};
\node[font=\fontsize{16}{16}] at (2.25, 4.5) (a) {$D_{4}$};
\node[font=\fontsize{16}{16}] at (3.75, 2.5) (a) {$D_{5}$};
\node[font=\fontsize{16}{16}] at (5.25, 1) (a) {$D_{6}$};
\end{tikzpicture}
\end{minipage}
\begin{minipage}{.24\textwidth}
\begin{tikzpicture}[scale=0.55, every node/.style={scale=0.4}]
\draw[black, thick](0,0)--(6,0);
	\draw[black, thick] (0,0)--(0,8);
\draw[black, thick] (0,8)--(6,8);
\draw[black, thick](6,0)--(6,8);
\draw[black, thick](0,3)--(6,3);
\draw[black, thick](3,8)--(3,0);
\draw[black, thick](1.5,8)--(1.5,3);
\draw[black, thick](0,6)--(3,6);
\draw[black, thick](4.5,0)--(4.5,3);
\draw[black, thick](3,2)--(6,2);
\draw[black, thick](.75,8)--(.75,6);
\draw[black, thick](2.25,6)--(2.25,3);
\draw[black, thick](3.75,3)--(3.75,2);
\draw[black, thick](5.25,2)--(5.25,0);
\draw[black, thick](0,7)--(1.5,7);
\draw[black, thick](1.5,4)--(3,4);
\draw[black, thick](3,3)--(4.5,3);
\draw[black, thick](4.5,1)--(6,1);  
\draw[black, thick](3,2.5)--(3.75,2.5); 

\node[font=\fontsize{18}{18}] at (4.5,5.5) (a) {$A_{1,2}$};
\node[font=\fontsize{18}{18}] at (1.5,1.5) (a) {$A_{2,1}$};
\node[font=\fontsize{16}{16}] at (2.25,7) (a) {$A_{3,4}$};
\node[font=\fontsize{16}{16}] at (.75,5) (a) {$A_{4,3}$}; 
\node[font=\fontsize{16}{16}] at (5.25,2.5) (a) {$A_{5,6}$};
\node[font=\fontsize{16}{16}] at (3.75,1) (a) {$A_{6,5}$}; 
\node[] at (1.15, 7.5) (a) {$A_{7,8}$}; 
\node[] at (0.35, 6.5) (a) {$A_{8,7}$}; 
\node[] at (2.65, 5) (a) {$A_{9,10}$}; 
\node[] at (1.9, 3.5) (a) {$A_{10,9}$};
\node[] at (4.1, 2.75) (a) {$A_{11,12}$};
\node[] at (3.4, 2.25) (a) {$A_{12,11}$}; 
\node[] at (5.6, 1.5) (a) {$A_{13,14}$}; 
\node[] at (4.9, 0.5) (a) {$A_{14,13}$}; 
\filldraw [fill=gray!30] (0,7) rectangle (.75,8);
\filldraw [fill=gray!30] (.75,6) rectangle (1.5,7);
\filldraw [fill=gray!30] (1.5,4) rectangle (2.25,6);
\filldraw [fill=gray!30] (2.25,3) rectangle (3,4);
\filldraw [fill=gray!30] (3,2.5) rectangle (3.75,3);
\filldraw [fill=gray!30] (3.75,2) rectangle (4.5,2.5);
\filldraw [fill=gray!30] (4.5,1) rectangle (5.25,2);
\filldraw [fill=gray!30] (5.25,0) rectangle (6,1);
\node[] at (.35, 7.5) (a) {$D_{7}$};
\node[] at (1.15, 6.5) (a) {$D_{8}$};
\node[] at (1.9, 5.5) (a) {$D_{9}$};
\node[] at (2.65, 3.5) (a) {$D_{10}$};
\node[] at (3.4, 2.75) (a) {$D_{11}$};
\node[] at (4.1, 2.25) (a) {$D_{12}$};
\node[] at (4.9, 1.5) (a) {$D_{13}$};
\node[] at (5.6, .5) (a) {$D_{14}$};
\end{tikzpicture}
\end{minipage}
\caption{A HODLR matrix is recursively partitioned into off-diagonal blocks associated with the tree $\tree$ (\emph{left}).
  Blocking associated with the first two levels of the tree are shown in the \emph{middle}, and blocking associated with first three levels is shown on the \emph{right}. Each of the off-diagonal blocks \smash{$A_{\ell r}$} is low rank. If the conditions in~\cref{eq:hss_diagonal_recursion} and~\cref{eq:hss_generator_recursion} also hold, we say $H$ has HSS structure.}
\label{fig:HODLRandTree}
\end{figure}
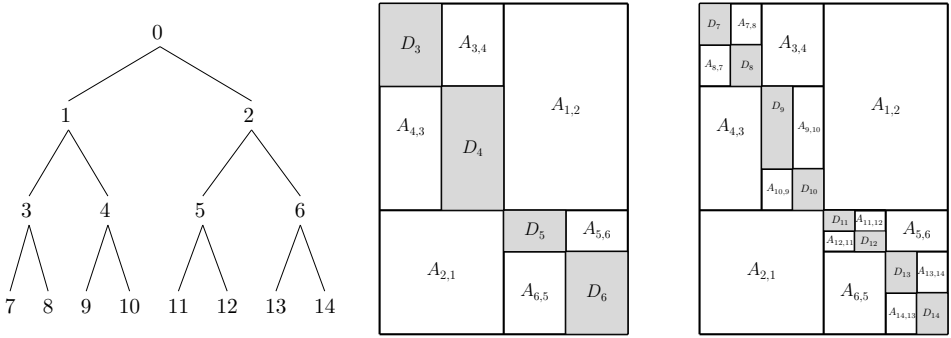

\subsection{HSS matrices} \label{sec:HSS}
In the HODLR form, each diagonal block \smash{$D_t$} (corresponding to a parent $t$ with children $\ell$ and $r$) has the recursive form
\begin{equation*}
  D_t = \twobytwo{D_\ell}{A_{\ell r}}{A_{r\ell}}{D_r}.
\end{equation*}
and we assume no relation between the off-diagonal blocks $A_{\ell r}$ and $A_{r\ell}$ across different levels of $\tree$.
A HODLR matrix is said to have \emph{HSS structure}~\cite{chandrasekaran2006fast} if the off-diagonal blocks are related in the following (recursive) way:
If $t\in\tree$ is not a root or leaf and $t$ has children $\ell$ and $r$, then
\begin{equation} \label{eq:hss_diagonal_recursion}
  D_t = \twobytwo{D_\ell}{A_{\ell r}}{A_{r\ell}}{D_r} = \twobytwo{D_\ell}{U_\ell B_{\ell r} V_r^*}{U_r B_{r \ell} V_\ell^*}{D_r}, 
\end{equation}
where the the \emph{basis matrices} \smash{$U_t$} and \smash{$V_t$} are related across levels of the tree by relations
\begin{equation} \label{eq:hss_generator_recursion}
  U_t = \twobytwo{U_\ell}{0}{0}{U_r} \twobyone{R_{\ell t}}{R_{r t}}, \quad V_t = \twobytwo{V_\ell}{0}{0}{V_r} \twobyone{W_{\ell t}}{W_{r t}}.
\end{equation}
This structure is analogous to the structure of the standard square HSS matrix (closely related to the \smash{$\mathcal{H}^2$} structure~\cite{hackbusch2015hierarchical}), except that the diagonal blocks are rectangular with  dimensions determined via the HODLR tree for $C$, which in turn is based on the clusters associated with $C$. The recursions \cref{eq:hss_diagonal_recursion} and \cref{eq:hss_generator_recursion} allow the entire matrix \smash{$C = D_0$} to be assembled from the following \emph{generator} matrices:
\begin{itemize}
\item Leaf generators: \smash{$D_t$}, \smash{$U_t$}, and \smash{$V_t$} for each leaf node $t\in\tree$.
\item Parent-child generators: \smash{$R_{c t}$} and \smash{$W_{c t}$} for every parent $t\in\tree$ with child $c$.
\item Sibling generators: \smash{$B_{\ell r}$} for every pair of siblings $\ell$ and $r$.
\end{itemize}

\subsubsection{HSS rows and columns}
The construction of the generator matrices requires finding \smash{$U_t$} and \smash{$V_t$} at each level so that they are \emph{nested} in the sense that if $t$ has children $\ell$ and $r$, the first rows of \smash{$U_t$} are spanned by \smash{$U_\ell$}, with the rest of \smash{$U_t$}'s rows are spanned by \smash{$U_r$}. To do this, we work with the \emph{HSS rows} and \emph{HSS columns} of the matrix $C$.

An HSS column of $C$ associated with $ t \in \tree$ is formed by selecting a strip of contiguous columns \smash{$K_t$} and excluding the rows in \smash{$J_t$}. We denote this as \smash{$A_t^{\rm col} = C(J_t^c, K_t)$} where \smash{$J_t^c = J_0 \setminus J_t$} denotes the complement of \smash{$J_t$} defined in \cref{sec:cHODLR}.
One can similarly define the HSS row \smash{$A_t^{\rm row} = C(J_t, K_t^c)$}. \Cref{fig:HSSrows_cols} displays some HSS rows and columns.

If the maximum rank of any HSS block row or column in $C$ is $k$,  then there exists a set of off-diagonal generators $\{U_t,V_t,R_{ct},W_{ct},B_{\ell r}\}$ where each generator has at most $k$ columns. By storing only the generators, the HSS matrix can be represented using only $\mathcal{O}((m+n)k)$ storage. Using an interpolative decomposition method~\cite{cheng2005compression} further reduces the storage cost since each \smash{$B_{\ell r}$} is a submatrix of $C$ and can therefore be stored via a small list of indices.  In practice, the rank of the off-diagonal leaf generator matrices are selected adaptively, so $k$ can vary for each leaf node. 

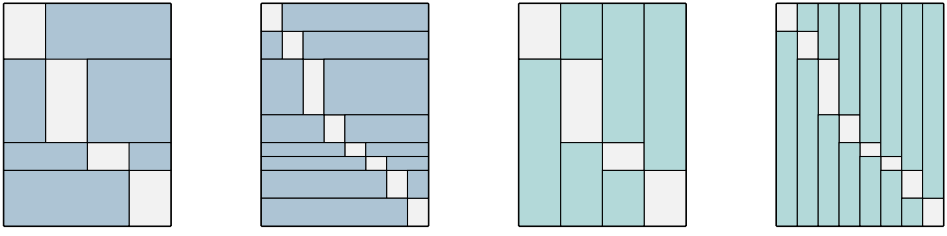
\begin{figure}
\centering
\begin{tikzpicture}[scale=0.37, every node/.style={scale=0.4}]
\draw[black, thick](0,0)--(6,0);
	\draw[black, thick] (0,0)--(0,8);
\draw[black, thick] (0,8)--(6,8);
\draw[black, thick](6,0)--(6,8);
\filldraw [fill=gray!10] (0,6) rectangle (1.5,8); 
\filldraw [fill=gray!10] (1.5,3) rectangle (3,7);
\filldraw [fill=gray!10] (3,2) rectangle (5,3);
\filldraw [fill=gray!10] (4.5,0) rectangle (6,2); 
\filldraw [fill=myblue!50] (1.5, 6) rectangle (6, 8); 
\filldraw [fill=myblue!50] (0, 3) rectangle (1.5, 6); 
\filldraw [fill=myblue!50] (3, 3) rectangle (6, 6); 
\filldraw [fill=myblue!50] (0, 2) rectangle (3, 3); 
\filldraw [fill=myblue!50] (4.5, 2) rectangle (6, 3);
\filldraw [fill=myblue!50] (0, 0) rectangle (4.5, 2); 
\end{tikzpicture}
\hspace{1cm}
\begin{tikzpicture}[scale=0.37, every node/.style={scale=0.4}]
\draw[black, thick](0,0)--(6,0);
	\draw[black, thick] (0,0)--(0,8);
\draw[black, thick] (0,8)--(6,8);
\draw[black, thick](6,0)--(6,8);
\filldraw [fill=gray!10] (0,7) rectangle (.75,8); 
\filldraw [fill=gray!10] (.75,6) rectangle (1.5,7);
\filldraw [fill=gray!10] (1.5,4) rectangle (2.25,6);
\filldraw [fill=gray!10] (2.25,3) rectangle (3,4); 
\filldraw [fill=gray!10] (3,2.5) rectangle (3.75,3); 
\filldraw [fill=gray!10] (3.75,2) rectangle (4.5,2.5); 
\filldraw [fill=gray!10] (4.5,1) rectangle (5.25,2); 
\filldraw [fill=gray!10] (5.25,0) rectangle (6,1);
\filldraw [fill=myblue!50] (.75, 7) rectangle (6, 8);
\filldraw [fill=myblue!50] (0, 6) rectangle (.75, 7); 
\filldraw [fill=myblue!50] (1.5, 6) rectangle (6, 7); 
\filldraw [fill=myblue!50] (0, 4) rectangle (1.5, 6); 
\filldraw [fill=myblue!50] (2.25, 4) rectangle (6, 6); 
\filldraw [fill=myblue!50] (0, 3) rectangle (2.25, 4); 
\filldraw [fill=myblue!50] (3, 3) rectangle (6, 4); 
\filldraw [fill=myblue!50] (0, 2.5) rectangle (3, 3); 
\filldraw [fill=myblue!50] (3.75, 2.5) rectangle (6, 3); 
\filldraw [fill=myblue!50] (0, 2) rectangle (3.75, 2.5); 
\filldraw [fill=myblue!50] (4.5, 2) rectangle (6, 2.5); 
\filldraw [fill=myblue!50] (0, 1) rectangle (4.5, 2); 
\filldraw [fill=myblue!50] (5.25, 1) rectangle (6, 2); 
\filldraw [fill=myblue!50] (0, 0) rectangle (5.25, 1); 
\end{tikzpicture}
\hspace{1cm}
\begin{tikzpicture}[scale=0.37, every node/.style={scale=0.4}]
\draw[black, thick](0,0)--(6,0);
	\draw[black, thick] (0,0)--(0,8);
\draw[black, thick] (0,8)--(6,8);
\draw[black, thick](6,0)--(6,8);
\filldraw [fill=gray!10] (0, 6) rectangle (1.5, 8); 
\filldraw [fill=gray!10] (1.5,3) rectangle (3,6);
\filldraw [fill=gray!10] (3,2) rectangle (4.5, 3); 
\filldraw [fill=gray!10] (4.5, 0) rectangle (6, 2);
\filldraw [fill=teal!30] (0,0) rectangle (1.5, 6); 
\filldraw [fill=teal!30] (1.5, 0) rectangle (3, 3); 
\filldraw [fill=teal!30] (1.5, 6) rectangle (3, 8); 
\filldraw [fill=teal!30] (3,0) rectangle (4.5, 2); 
\filldraw [fill=teal!30] (3,3) rectangle (4.5, 8); 
\filldraw [fill=teal!30] (4.5, 2) rectangle (6, 8); 
\end{tikzpicture}
\hspace{1cm}
\begin{tikzpicture}[scale=0.37, every node/.style={scale=0.4}]
\draw[black, thick](0,0)--(6,0);
	\draw[black, thick] (0,0)--(0,8);
\draw[black, thick] (0,8)--(6,8);
\draw[black, thick](6,0)--(6,8);
\filldraw [fill=gray!10] (0,7) rectangle (.75,8); 
\filldraw [fill=gray!10] (.75,6) rectangle (1.5,7);
\filldraw [fill=gray!10] (1.5,4) rectangle (2.25,6);
\filldraw [fill=gray!10] (2.25,3) rectangle (3,4); 
\filldraw [fill=gray!10] (3,2.5) rectangle (3.75,3); 
\filldraw [fill=gray!10] (3.75,2) rectangle (4.5,2.5); 
\filldraw [fill=gray!10] (4.5,1) rectangle (5.25,2); 
\filldraw [fill=gray!10] (5.25,0) rectangle (6,1);
\filldraw [fill=teal!30] (0,0) rectangle (.75, 7);
\filldraw [fill=teal!30] (.75, 7) rectangle (1.5, 8); 
\filldraw [fill=teal!30] (.75, 0) rectangle (1.5, 6); 
\filldraw [fill=teal!30] (1.5, 6) rectangle (2.25, 8); 
\filldraw [fill=teal!30] (1.5,0) rectangle (2.25, 4); 
\filldraw [fill=teal!30] (2.25,4) rectangle (3, 8); 
\filldraw [fill=teal!30] (2.25, 0) rectangle (3, 3); 
\filldraw [fill=teal!30] (3, 3) rectangle (3.75, 8); 
\filldraw [fill=teal!30] (3,0) rectangle (3.75, 2.5); 
\filldraw [fill=teal!30] (3.75, 2.5) rectangle (4.5, 8); 
\filldraw [fill=teal!30] (3.75, 0) rectangle (4.5, 2); 
\filldraw [fill=teal!30] (4.5, 2) rectangle (5.25, 8); 
\filldraw [fill=teal!30] (4.5, 0) rectangle (5.25, 1); 
\filldraw [fill=teal!30] (5.25, 1) rectangle (6, 8); 
\end{tikzpicture}
\caption{Example HSS rows (blue, \emph{left}) and HSS columns (teal, \emph{right}) of $C$ are shown for two different levels of $\tree$. The gray rectangular blocks are excluded from the submatrices.}
\label{fig:HSSrows_cols}
\end{figure}

With these definitions in place, we can describe the rank structure of the transformed NUDFT matrix \smash{$C = VF^*$}, along with an algorithm that constructs the HSS generators. We then describe a solver based on these generators in~\Cref{sec:hier-semis-least}. 

\subsection{The rank structure of $C$}
\label{sec:rank-structure-c} The next theorem shows\footnote{The authors express gratitude to Bernhard Beckermann, who independently derived a proof similar in structure to this one in the context of Toeplitz matrices in 2018. Work related to this is found in~\cite[Ch. 4]{wilber2021computing}.} that  the $\epsilon$-ranks of the HSS columns and rows of $C$ only grow like \smash{$\mathcal{O}(\log( n) \log (1/\epsilon) )$}. The $\epsilon$-ranks of every off-diagonal HODLR block in $C$ behave similarly.

\begin{theorem} \label{thm:cHSS_col} Let \smash{$C \in \mathbb{C}^{m \times n}$} be as in~\cref{eq:diag_sylv}. 
If $A$ is an HSS column or row of $C$ (or a submatrix of such a column or row), then
\begin{equation}  \label{eq:erank_col}
\erank(A) \leq  \left \lceil \frac{2\log(4/\epsilon) \log( 4n )}{\pi^2} \right \rceil .
\end{equation}
\end{theorem}

\begin{proof}
  The proofs for HSS rows and columns are similar, so we consider the case of an HSS column.
  Consider an HSS column \smash{$A^{\rm col}_t = C_{J_t^c K_t}$}, where
  \begin{equation*}
    K_t = \{ \kappa_1, \kappa_1 + 1,\ldots,\kappa_2 \} \quad \text{and} \quad J_t^c = \{ 1\le j \le m : j \notin \cluster_\kappa \text{ for every } \kappa_1\le \kappa\le \kappa_2 \}.
  \end{equation*}
  Let \smash{$A = C_{JK}$} be any submatrix of \smash{$A^{\rm col}_t$} where \smash{$K\subset K_t$} and \smash{$J\subset J_t^c$}. Let \smash{$\Gamma_J = \diag(\gamma_j )_{j \in J}$} and \smash{$\Lambda_K = \diag(\omega^{2\kappa})_{\kappa \in K}$}.
  By \cref{eq:diag_sylv}, we observe that 
  $$ \rank( \Gamma_J A - A\Lambda_K) = 1.$$
  From~\cref{eq:zolobound_svs}, we have that for $1 \leq k \leq \min(|J|,|K|) - 1$, 
  $$ \sigma_{k + 1}(A) \leq Z_k(\arc_J, \arc_K)\|A\|_2, $$
  where \smash{$\arc_J$} and \smash{$\arc_K$} are arcs on the unit circle that contain \smash{$\lambda(\Gamma_J)$} and \smash{$\lambda(\Lambda_K)$}, respectively:
  \begin{align*} 
    \arc_J &= \{e^{i\theta} : \pi(2\kappa_2+1)/n \le \theta \le \pi(2n+2\kappa_1-1)/n\},\\
    \arc_K &= \{e^{i\theta} : 2\pi \kappa_1/n \le \theta \le 2\pi \kappa_2/n \}.
  \end{align*}
  Applying \cref{lemma:arcs}, we have that
  \begin{equation} \label{eq:bound_svA}
    \sigma_{k + 1}(A) \leq 4 \mu^{-2k} \|A\|_2, \quad \mu = \exp\left( \frac{\pi^2}{2 \log (16 \eta) } \right), 
  \end{equation}
  where 
  \begin{equation} \label{eq:eta}
    \eta = \frac{ \sin^2\big(\tfrac{\pi}{n}(\kappa_2-\kappa_1)+\tfrac{\pi}{2n}\big) }{\sin^2\big( \tfrac{\pi}{2n}\big) } \le n^2.
  \end{equation}
  The upper bound on $\eta$ follows from the fact that  $\sin (\pi/2n) \geq 1/n$ for $n \ge 1$.
  Substituting \cref{eq:eta} into \cref{eq:bound_svA}, rearranging, and using the definition of $\epsilon$-rank establishes \cref{eq:erank_col}.
\end{proof}

We remark that in some extreme cases, an HSS row or column may have dimensions such that its mathematical rank  is less than the bound given in \cref{thm:cHSS_col} on the $\epsilon$-rank. Also, tighter bounds can be established for the submatrices of \smash{$A_t^{\rm col}$} (or \smash{$A_t^{\rm row}$}), including HODLR blocks of $C$, if one works directly with the parameters defining the particular submatrices. 

\subsection{Constructing the generators}
\label{sec:constr-gener}

\cref{thm:cHSS_col} establishes that the HSS rows and columns have low numerical rank, which implies that there exists a matrix $H\approx C$ for which each HSS row or column of $H$ is a low rank approximation to the corresponding submatrix of $C$.
We use a fADI-based interpolative decomposition method~\cite[Ch.4]{wilber2021computing} to cheaply construct the generators associated with $H$. We illustrate the idea using an HSS row. Let $t$ be a leaf node of $\tree$ and let \smash{$A_t^{\rm row}$} be the $t$th HSS row of $C$. Assume \smash{$A_t^{\rm row}$} has \smash{$m_t$} rows and \smash{$n\! - \!n_t$} columns, with \smash{$n_t \sim \log n$}. We assume at this level that  \smash{$m_t$} is $\mathcal{O}(\log m)$. Applying fADI directly (see \cref{sec:fadi_alg}), one would construct the low rank approximation \smash{$ZW^* \approx A_t^{\rm row}$}, where $Z$ is a small matrix of size \smash{$m_t \times k$},  and $W$ is tall and thin, of size \smash{$n\! - \!n_t \times k$}. However, it is wasteful and costly to construct $W$. Fortunately, fADI constructs $Z$ and $W$ independently~\cite[Ch.4]{wilber2021computing}, and there is no need to construct $W$ or touch the long dimension of \smash{$A_t^{\rm row}$} at all in order to construct $Z$. With $Z$ in hand, we do a column-pivoted QR decomposition of $Z^T$ to find 
$$ Z = \mathcal{P} \mathcal{R}^T\mathcal{Q}^T, \quad R^T = \twobyone{R_{\rm a}^T}{R_{\rm b}^T}$$
where $\mathcal{P}$ is a permutation matrix, \smash{$\mathcal{R}^T \in \mathbb{C}^{m_t \times k}$} is lower triangular with $k$ columns, \smash{$\mathcal{Q} \in \mathbb{C}^{k \times k}$} is unitary, and $R_{\rm a}^T$ is of size $k \times k$.  This decomposition is then used in a standard way to construct the one-sided interpolative decomposition~\cite{cheng2005compression,martinsson2019fast}
$$ A_t^{\rm row} = U_t A_t^{\rm row}(S_t^{\rm row},\; :\; ), \quad U_t = \mathcal{P} \twobyone{I_k}{R_{\rm b}^TR_{\rm a}^{-T}},$$
where \smash{$I_k$} is the $k \times k$ identity matrix and \smash{$S_t^{\rm row}$} indexes a subset of rows corresponding to the top $k$ rows selected by $\mathcal{P}$. We store the leaf generator basis matrix \smash{$U_t$} and the index set \smash{$S_t^{\rm row}$}. We call the subselected rows indexed by \smash{$S_{t}^{\rm row}$} \textit{basis rows}. A similar process is applied to HSS columns to construct each of the leaf generator basis matrices \smash{$V_t$} and select basis columns indexed by \smash{$S_t^{\rm col}$}. The sibling generators at this level are then simply given by \smash{$B_{\ell r} = C(S_\ell^{\rm row}, S_r^{\rm col})$}.

Once the leaf generators are constructed, we move one level up $\mathcal{T}$. For each parent $t \in \mathcal{T}$ at this level, we must construct the left and right parent-child generators \smash{$R_{\ell t}$} (\smash{$R_{r t}$} resp.) and \smash{$W_{\ell t}$} (\smash{$W_{r t}$} resp.), as well as sibling generators. 
We do this by applying fADI-based one-sided interpolative decompositions on the basis rows or basis columns of the appropriate children, and note that then number of basis rows (columns) is $< 2k$, where $k$ is the maximum rank of any HSS row or column at the leaf level.  In total, the interpolative fADI-based compression strategy constructs the generators for $H \approx C$ in \smash{$\mathcal{O} (m k^2) =\mathcal{O} (m \log^2(n) \log^2(1/\epsilon))$} operations.  With $H \approx C$ constructed, we have completed step 2 in \Cref{alg:generalsolver}. We note that while it is advantageous to leverage displacement structure via ADI, one could also use our bounds and apply a different core low rank approximation method to do the compression, such as one based on randomized numerical linear algebra~\cite{martinsson2011fast}.  We now turn to Step 3 in \cref{alg:generalsolver} and develop an efficient rank-structured solver for $Hy = b$.

\section{A hierarchically semiseparable least-squares solver}
\label{sec:hier-semis-least}
We now present an algorithm that approximately solves the problem
\begin{equation} \label{eq:LS}
  y = \argmin_{y\in \complex^n} \norm{Hy - b}_2^2,
\end{equation}
The developments in this section are general, applying to an arbitrary HSS matrix \smash{$H \in \complex^{m\times n}$} and vector \smash{$b \in \complex^m$}.
Throughout this section, let $k$ be the maximum rank of the HSS rows and columns of $H$.
For the inverse NUDFT problem, we demonstrated in \cref{sec:rank_structs} that one can construct an HSS matrix $H$ where $C \approx H$ with \smash{$k = \order(\log(n) \log(1/\epsilon))$}.
Up to the approximation $H\approx C$, the least-squares problem \cref{eq:LS} is equivalent to the inverse NUDFT problem \cref{eq:least_squares} with $y = Fx$.

\subsection{Solutions via the normal equations} \label{sec:normal}
A natural first approach to the least-squares problem \cref{eq:LS} is to work with the normal equations \smash{$(H^*H)y = H^*b$}, a linear system involving the square HSS matrix \smash{$H^*H$}. Indeed, the maximum rank of an HSS row or column in \smash{$H^*H$} is $2k$ and the HSS generators for \smash{$H^*H$} can be found in \smash{$\mathcal{O}(mk^2)$} operations. At that point, any number of fast direct solvers~\cite{rouet16dist,massei2020hm,xia10hier} can be applied to solve the normal equations. Relatedly, one may start with $Vx =b$ and observe that \smash{$V^*V$} is Toeplitz. Fast direct solvers for Toeplitz matrices are available~\cite{wilber2021computing, xia2012superfast}, so one must only implicitly form \smash{$V^*V$} in such a way that these solvers can be applied to solve~\cref{eq:Vandermonde_sys} via the normal equations. We make comparisons with such an algorithm in our numerical experiments in \cref{sec:experiments}.

The problem with these methods is that they square the condition number of the problem: solving the normal equations explicitly can result in lower accuracy than a dedicated least-squares solver. In our setting, we expect that $\kappa_2(H)$ may already be large and want to avoid exacerbating the issue. 

\subsection{A direct least-squares solver} \label{sec:solver}
Due to the above issues, we avoid the normal equations and work directly with $H$. Our algorithm is based on the URV factorization algorithm introduced in~\cite{xi2014superfast} for solving least-squares problems involving Toeplitz matrices. While that solver takes advantage of additional structure related to Toeplitz matrices, its URV factorization method is generalizable, and we present such a generalization here. The URV factorization is so-named because it decomposes a rectangular HSS matrix into the factors $U$, $R$, and $V$, with $U$ and $V$ being a product of unitary matrices and $R$ being upper triangular. These factors are never constructed or stored explicitly, but can instead be represented by sequences of sparse blocks in a hierarchical structure. The blocks are then used to reduce the linear system into a collection of small triangular systems that can be solved efficiently. The method is analogous to solution of least-squares problems by QR factorization, but is modified to exploit hierarchical structures. 

An important difference between the present work and~\cite{xi2014superfast} is the size of the blocks. In~\cite{xi2014superfast}, the rectangular diagonal blocks at each level of $\mathcal{T}$ are uniform or close to uniform in size. In our case, the block sizes are determined by the clustering patterns of the nodes. They can vary substantially from vertex to vertex in $\tree$, and short-fat rectangular diagonal blocks inevitably arise\footnote{The algorithm is breakdown-free as long as each block is of size \smash{$\tilde{m}_k \times n_t$}, with \smash{$n_t-k \le \tilde{m}_k$}. This is a significantly more mild condition than that all of the blocks are thin, i.e., \smash{$n_t \leq m_k$}).}.  While our method for constructing $H$ relies on the displacement structure of $C$, our solver is completely general and can be used for any HSS matrix.

An essential ingredient in our solver is the \emph{size reduction} step, introduced by~\cite{xi2014superfast}. The size reduction step applies unitary transformations to introduce additional zeros at each stage of the algorithm, controlling the \emph{rectangularity} (ratio of rows to columns) during recursive stages of the algorithm. We now describe the algorithm, beginning with URV factorization. 

\subsection{Factorization}
\label{sec:factorization}

The algorithm proceeds by traversing the HSS tree beginning with the leaves and continuing until the root is reached. Before diving into the algorithm, we establish some notation.
For each $t\in \tree$, the \emph{$t$th block row} are the rows of $H$ indexed by \smash{$J_t$} and the \emph{$t$th block column} are the columns indexed by \smash{$K_t$}.
Observe that, the $t$th block row contains the diagonal block \smash{$D_t$} and the HSS row \smash{$A^{\rm row}_t$} as complementary submatrices, and likewise for the $t$th block column.
For a leaf $t\in\tree$, we denote \smash{$m_t$} (resp.\ \smash{$n_t$}) as the number of rows (resp.\ columns) in the $t$th block row (resp.\ column).
As a notational convenience, we assume all generators (except \smash{$D_t$}) have exactly $k$ columns; this restriction is not necessary in practice.

Begin at each leaf $t\in \tree$. 
The first step is the \emph{size reduction} step, designed to control the number of rows relative to columns in subsequent steps of the algorithm. The size reduction step introduces zeros at the bottom of \smash{$U_t$} and \smash{$D_t$} by means of a QR decomposition
\begin{equation*}
  \onebytwo{U_t}{D_t} = \Omega_t \begin{blockarray}{ccc} \matrixsize{k} & \matrixsize{n_t} & \\ \begin{block}{[cc]c} \tilde{U}_t & \tilde{D}_t & \matrixsize{k+n_t} \\ 0 & 0 & \matrixsize{m_t - k - n_t} \\\end{block} \end{blockarray}, 
\end{equation*}
where \smash{$\Omega_t\in\complex^{(k+n_t)\times (k+n_t)}$} is unitary. 
When we multiply the $t$th block row of $H$ by \smash{$\Omega^*_t$}, it zeroes out many of the rows in that block, allowing us to ignore these rows in future steps of the algorithm. 
This size reduction step only needs to be performed if \smash{$\onebytwo{U_t}{D_t}$} has many more rows than columns; using the factor $6$ as cutoff produces good results in our experiments.
We denote the number of rows of \smash{$\tilde{U}_t,\tilde{D}_t$} as \smash{$\tilde{m}_t$}.

We now work with the reduced leaf generators \smash{$\tilde{U}_t$, $\tilde{D}_t$}, and \smash{$V_t$} to compute the URV factorization as in~\cite{xi2014superfast}.  Useful diagrams related to this process in the Toeplitz setting are given in~\cite[Fig.~2.3]{xi2014superfast}. The first step introduces zeros into \smash{$V_t$} by means of a \emph{reversed QR factorization}:
\begin{equation*}
  V_t = P_t
  \begin{blockarray}{cc}
    \matrixsize{n_t} \\
    \begin{block}{[c]c}
      0 & \matrixsize{n_t - k} \\
      \overline{V}_t & \matrixsize{k} \\
    \end{block}
  \end{blockarray},
\end{equation*}
where \smash{$\overline{V}_t = \left[\opentriangle\right]$} is \emph{antitriangular}, \smash{$(\overline{V}_t)_{j\kappa} = 0$} if $j+\kappa < k+1$, and \smash{$P_t \in \complex^{n_t\times n_t}$} is unitary.
Such a factorization can be obtained by computing a QR factorization of $V_t$, reversing the order of the columns of the first factor, and reversing the order of the rows of the second factor.
Multiplying the $t$th block column by \smash{$P_t$} has the effect of zeroing out many columns of the HSS column \smash{$A^{\rm col}_t$}.
The modified diagonal block \smash{$\hat{D}_t = \tilde{D}_t P_t$} is dense, so we introduce as many zeros as possible by a partial triangularization.
Partition the diagonal block as:
\begin{equation*}
  \hat{D}_t = \tilde{D}_t P_t =
  \begin{blockarray}{ccc}
    \matrixsize{n_t-k} & \matrixsize{k} \\
    \begin{block}{[cc]c}
      \hat{D}_{t;11} & \hat{D}_{t;12} & \matrixsize{n_t - k} \\
      \hat{D}_{t;21} & \hat{D}_{t;22} & \matrixsize{\tilde{m}_t - n_t + k}\\
    \end{block}
  \end{blockarray}.
\end{equation*}

Now, introduce zeros in \smash{$\hat{D}_t$} by means of a QR factorization:
\begin{equation*}
  \twobyone{\hat{D}_{t;11}}{\hat{D}_{t;21}} = Q_t
  \begin{blockarray}{cc}
    \matrixsize{n_t - k}\\
    \begin{block}{[c]c}
      \overline{D}_{t;11} & \matrixsize{n_t - k} \\
      0 & \matrixsize{\tilde{m}_t - n_t + k}\\ 
    \end{block}
  \end{blockarray},
\end{equation*}
where $Q_t\in \complex^{\tilde{m}_t\times \tilde{m}_t}$ is unitary.
Now multiply \smash{$\hat{D}_t$} by \smash{$Q_t^*$}. This results in a modified diagonal block
\begin{equation*}
  \overline{D}_t = Q_t^* \hat{D}_t =   \begin{blockarray}{ccc}
    \matrixsize{n_t-k} & \matrixsize{k} \\
    \begin{block}{[cc]c}
      \overline{D}_{t;11} & \overline{D}_{t;12} & \matrixsize{n_t - k} \\
      0 & \overline{D}_{t;22} & \matrixsize{\tilde{m}_t - n_t + k} \\
    \end{block}
  \end{blockarray}
\end{equation*}
and modified leaf generator matrix
\begin{equation*}
  \overline{U}_t= Q_t^* \tilde{U}_t = \begin{blockarray}{cc}
    \matrixsize{k} \\
    \begin{block}{[c]c}
      \overline{U}_{t;1} & \matrixsize{n_t - k} \\
      \overline{U}_{t;2} & \matrixsize{\tilde{m}_t - n_t + k}\\
    \end{block}
  \end{blockarray}.
\end{equation*}

The new diagonal block \smash{$\overline{D}_t$} is now almost triangular. 
The sticking point is the \hbox{$(2,2)$-block} \smash{$\overline{D}_{t;22}$}. If we truncate the $t$th block row so that the rows coinciding with \smash{$\overline{D}_{t;22}$} are not included, then the remaining block row is upper triangular. This completes the factorization at the leaf level.  We then handle \smash{$\overline{D}_{t;22}$} recursively. 

For a non-leaf $t\in\tree$ with children $\ell$ and $r$, define
\begin{align*}
  D_t' = \twobytwo{\overline{D}_{\ell;22}}{\overline{U}_{\ell;2}B_{\ell,r}\overline{V}_r^*}{\overline{U}_{r;2}B_{r,\ell}\overline{V}_\ell^*}{\overline{D}_{\ell;22}}, \quad
  U_t' = \twobyone{\overline{U}_{\ell;2}R_{\ell,t}}{\overline{U}_{r;2}R_{r,t}}, \quad
  V_t' = \twobyone{\overline{V}_{\ell}W_{\ell,t}}{\overline{V}_{r}W_{r,t}},
\end{align*}
Here, \smash{$B_{\ell,r}$, $B_{r,\ell}$, $R_{\ell,t}$, $R_{r,t}$, $W_{\ell,t}$}, and \smash{$W_{r,t}$} refer to the sibling and parent-child generators for the HSS matrix $H$, defined in \cref{sec:HSS}.
If $t$ is not the root, we perform the same sequence of operations we performed for the leaf nodes, using \smash{$D_t'$, $U_t'$}, and \smash{$V_t'$} in place of \smash{$D_t$}, \smash{$U_t$}, and \smash{$V_t$} using \smash{$m_t$} and \smash{$n_t$} to denote the number of rows and columns of \smash{$D_t'$}, respectively.
If $t$ is the root, we only need to perform a QR decomposition
\begin{equation*}
  D_0' = Q_0
  \begin{blockarray}{cc}
    \matrixsize{m_0} \\
    \begin{block}{[c]c}
      \overline{D}_{0;11} & \matrixsize{m_0} \\
      0 & \matrixsize{n_0 - m_0} \\
    \end{block}
  \end{blockarray},
\end{equation*}
where \smash{$Q_0 \in \complex^{m_0\times m_0}$} is unitary.
The matrix \smash{$D_0'$} is small, so this QR decomposition is inexpensive.
A summary description is shown in \cref{alg:factor}. Once the blocks comprising the URV factors are computed, we are ready to solve the least-squares problem. 

\begin{algorithm}[t!]
  \caption{URV factorization for HSS matrix $H$} \label{alg:factor}
  \begin{algorithmic}[1]
    \Procedure{URV}{$t$} \Comment{$t\in\tree$ a node in the HSS tree}
    \If{$t$ not leaf}
    \State $\ell \leftarrow \text{$t$.left}$, $r \leftarrow \text{$t$.right}$
    \State Perform recursive calls \Call{URV}{$\ell$} and \Call{URV}{$r$}
    \State $D_t \leftarrow \twobytwo{\overline{D}_{\ell;22}}{\overline{U}_{\ell;2}B_{\ell,r}\overline{V}_r^*}{\overline{U}_{r;2}B_{r,\ell}\overline{V}_\ell^*}{\overline{D}_{\ell;22}}, \: U_t \leftarrow \twobyone{\overline{U}_{\ell;2}R_{\ell,p}}{\overline{U}_{r;2}R_{r,p}},\: V_t \leftarrow \twobyone{\overline{V}_{\ell}W_{\ell,p}}{\overline{W}_{r}R_{r,p}}$
    \EndIf
    \If{$t$ is root}
    \State $\left(Q_t, \twobyone{\overline{D}_{t;11}}{0}\right) \leftarrow \Call{QR}{D_t}$
    \Else
    \If{$\Call{Size}{\onebytwo{U_t}{D_t},2} \ge \texttt{srf} \cdot \Call{Size}{U_t,1}$} \Comment{We used $\texttt{srf} = 6$}
    \State $\left( \Omega_t, \twobytwo{\tilde{U}_t}{\tilde{D}_t}{0}{0}\right) \leftarrow \Call{QR}{\onebytwo{U_t}{D_t}}$
    \Else
    \State $\Omega_t \leftarrow \begin{bmatrix} {\:} \end{bmatrix}$, $\tilde{D}_t \leftarrow D_t$, and $\tilde{U}_t \leftarrow U_t$
    \EndIf
    \State $(P_t,\mathtt{tmp}) \leftarrow \Call{QR}{V_t}$
    \State $P_t = P_t(:,\mathtt{end}:-1:1)$, $\overline{V}_t \leftarrow \mathtt{tmp}(\Call{Size}{V_t,2}:-1:1,:)$ \Comment{Reversed QR}
    \State $\twobytwo{\hat{D}_{t;11}}{\hat{D}_{t;12}}{\hat{D}_{t;21}}{\hat{D}_{t;22}} \leftarrow \tilde{D}_tP_t$
    \State $\left(Q_t,\twobyone{\overline{D}_{t;11}}{0}\right) \leftarrow \textsc{QR}\left(\twobyone{\hat{D}_{t;11}}{\hat{D}_{t;21}}\right)$
    \State $\twobyone{\overline{D}_{t;12}}{\overline{D}_{t;22}} \leftarrow Q_t^* \twobyone{\hat{D}_{t;12}}{\hat{D}_{t;22}},\quad \twobyone{\overline{U}_{t;1}}{\overline{U}_{t;2}} \leftarrow Q_t^* \tilde{U}_t$
  
    \EndIf
    \EndProcedure
  \end{algorithmic}
\end{algorithm}

\subsection{Solution}
\label{sec:solution}
We break the procedure for solving the least-squares problem into two large steps. The first involves transforming the right-hand side, and the the second uses backsubstitution to solve a sequence of hierarchical triangular systems.
A summary of the procedure is provided in \cref{alg:solve}.

\begin{algorithm}[t]
  \caption{Least squares solution for linear system $Hy = b$ from URV factorization} \label{alg:solve}
  \begin{algorithmic}[1]
    \Procedure{URVSolve}{$b$}
    \State \Call{SolvePart1}{0, $b$} \Comment{Begin recursion at root node $0$}
    \State \Return \Call{SolvePart2}{0}
    \EndProcedure
  \end{algorithmic}
  \vspace{0.5em}
  
  \begin{algorithmic}[1]
    \Procedure{SolvePart1}{$t$, $b$}
    \If{$t$ not leaf}
    \State $\ell \leftarrow t\text{.left}$, $r\leftarrow t\text{.right}$, and partition $b = \twobyone{b_\ell}{b_r}$
    \State Perform recursive calls \Call{SolvePart1}{$\ell$, $b_\ell$} and \Call{SolvePart1}{$r$, $b_r$}
    \State $b \leftarrow \twobyone{c_{\ell;2}}{c_{r;2}}$
    \EndIf
    \If{$t$ not root and $\Omega_t \ne \begin{bmatrix} {\:} \end{bmatrix}$}
    \State $b \leftarrow \Omega_t^* b, \: b \leftarrow b(1:\Call{Size}{Q_t,1})$
    \EndIf
    \State $\twobyone{c_{t;1}}{c_{t;2}} \leftarrow Q_t^* b$
    \EndProcedure
  \end{algorithmic}
  \vspace{0.5em}
  
  \begin{algorithmic}[1]
    \Procedure{SolvePart2}{$t$}
    \If{$t$ is root}
    \State $y_t \leftarrow \overline{D}_{t;11}^{-1} c_{t;1}$
    \Else
    \State $w_{t;1} \leftarrow \overline{D}_{t;11}^{-1}(c_{t;1} - \overline{D}_{t;12}y_{t;2} - \overline{U}_{t;1} z_t)$
    \State $y_t \leftarrow P_t \twobyone{y_{t;1}}{y_{t;2}}$
    \EndIf
    \If{$t$ not leaf}
    \State $\ell \leftarrow t\text{.left}$, $r\leftarrow t\text{.right}$, and partition $y_t = \twobyone{w_{\ell;2}}{w_{r;2}}$
    \State $z_\ell \leftarrow B_{\ell,r} \overline{V}_r^* w_{r;2} + R_{\ell,t} z_t$, $z_r \leftarrow B_{r,\ell}\overline{V}_\ell^* w_{\ell;2} + R_{r,t} z_t$
    \State $y \leftarrow \twobyone{\Call{SolvePart2}{\ell}}{\Call{SolvePart2}{r}}$
    \EndIf
    \State \Return $y$
    \EndProcedure
  \end{algorithmic}
\end{algorithm}

\subsubsection{Apply unitary operations $Q_t^*$ and $\Omega_t^*$ to the right-hand side} The URV triangularization process modifies the rows of the linear system, and we must account for that in the right-hand side.
For each \emph{leaf} $t$, let \smash{$b_t$} denote the block of the right-hand side of $b$ consisting of those rows associated with $t\in \tree$.
Beginning with each leaf node $t$, we define
\begin{equation*}
  c_t = \twobytwo{Q_t^*}{0}{0}{I} \Omega_t^* b_t =
  \begin{blockarray}{cc}
    \begin{block}{[c]c}
      c_{t;1} & \matrixsize{n_t - k} \\ c_{t;2} & \matrixsize{\tilde{m}_t - n_t + k} \\ c_{t;3} & \matrixsize{m_t - \tilde{m}_t} \\
    \end{block}
  \end{blockarray}.
\end{equation*}
If the size reduction step was omitted, then one only needs to multiply by \smash{$Q_t^*$}:
\begin{equation*}
  c_t = Q_t^* b_t =   \begin{blockarray}{cc}
    \begin{block}{[c]c}
      c_{t;1} & \matrixsize{n_t - k} \\ c_{t;2} & \matrixsize{\tilde{m}_t - n_t + k} \\
    \end{block}
  \end{blockarray}.
\end{equation*}

Next, consider each parent $t$ with children $\ell$ and $r$.
If $t$ is not the root, then we set \smash{$b_t = \onebytwo{c_{\ell;2}^T}{c_{r;2}^T}^T$} and apply the same procedure as we did the leaves.
If $t$ is the root, we simply set \smash{$c_t = \onebytwo{c_{\ell;2}^T}{c_{r;2}^T}^T$}.

\subsubsection{Solving the hierarchical triangular system} \label{sec:solvestep}
We now perform a top-down traversal of $\tree$.
Beginning with the root $0$, solve the triangular system
\begin{equation*}
  \overline{D}_{0;11} y_0 = c_{0;1}.
\end{equation*}
For the children $\ell$ and $r$ of $0$, partition the solution as \smash{$y_0 = \onebytwo{w_{\ell;2}^T}{w_{r;2}^T}^T$} for \smash{$w_{\ell;2},w_{r;2} \in \complex^k$} and set \smash{$z_\ell = B_{\ell,r} \overline{V}_r^* w_{r;2}$}, \smash{$z_r = B_{r,\ell}\overline{V}_\ell^* w_{\ell;2}$}.

For a non-parent node $t$, we solve the triangular system
\begin{equation*}
  \overline{D}_{t;11} w_{t;1} = c_{t;1} - \overline{D}_{t;12} w_{t;2} - \overline{U}_{t;1} z_t
\end{equation*}
and define
\begin{equation*}
  y_t = P_t \twobyone{w_{t;1}}{w_{t;2}}.
\end{equation*}
If $t$ is a non-leaf with children $\ell$ and $r$, then partition \smash{$y_t = \onebytwo{w_{\ell;2}^T}{w_{r;2}^T}^T$} for \smash{$w_{\ell;2},w_{r;2} \in \complex^k$} and set
\begin{equation*}
  z_\ell = B_{\ell,r} \overline{V}_r^* w_{r;2} + R_{\ell,t} z_t, \quad z_r = B_{r,\ell}\overline{V}_\ell^* w_{\ell;2} + R_{r,t} z_t.
\end{equation*}
Conclude by assembling $y$ from all \smash{$y_t$} at leaf nodes $t$.
Pseudocode for the entire solution step is provided in \cref{alg:solve}.

\subsection{Computational costs}
\label{sec:operation-count}

To get a representative operation count for the URV factorization, we make the assumption that  all the diagonal blocks \smash{$D_t$} at the leaf level have $2k$ columns and $2(m/n)k$ rows and that all non-diagonal HSS generators have exactly $k$ columns (in practice, the number of columns for each generator will vary and is determined adaptively according to \cref{thm:cHSS_col}). This gives an overall cost for the URV factorization step of
\begin{equation} \label{eq:op_counts}
    32mk^2 + 124 nk^2 \text{ operations}.
\end{equation}
This is on the same asymptotic order as the HSS factorization step (Step 2 in \cref{alg:generalsolver}), but involves larger constants and is the most expensive step in the algorithm. Once the URV factorization is computed, each subsequent triangular solve is cheaper, requiring just $\order(mk)$ operations. From \cref{thm:cHSS_col}, we have that $k = \mathcal{O}(\log n \log 1/\epsilon)$. The conclusion is that whenever all  clusters associated with $V$ obey \smash{$\mathcal{O}\left( |\cluster_\kappa | \right) = \mathcal{O}(k)$ }, the system $Vx = b$ can be solved with time complexity \smash{$\mathcal{O}\left( (m +n) \log^2 (n) \log^2(1/\epsilon) \right)$}.  Even if this condition is violated, the solver is relatively insensitive to the distribution of the nodes. In comparison, all the strategies we tested in \cref{sec:experiments} fail if there are substantial gaps between nodes (e.g., as in Grid  4 in \cref{fig:gridsandnodes}).

\subsection{Numerical stability}
While we have not undertaken a stability analysis of the exact procedure shown here, a full stability analysis for the structure-exploiting URV factorization method of \cite{xi2014superfast} comes equipped with a backward stability guarantee, up to a modest growth factor.
Given the design of the URV factorization algorithm using unitary operations, our numerical experience, and the results of \cite{xi2014superfast}, we conjecture that the procedure from this paper is also backward stable.

\subsection{Implementation details} To implement our ADI-based HSS factorization and also the least-squares solver, we have made use of the \texttt{hm-toolbox} package~\cite{massei2020hm} in MATLAB, modifying it to handle nonuniform partitioning and make use of ADI-based compression strategies. We have also  implemented the URV solver, which in principle could be applied to any rectangular HSS matrix. All of our code is available at \url{https://github.com/heatherw3521/NUDFT}, along with experiments and basic examples. While the underlying code is relatively complicated, using the solver is straightforward. It can be applied in MATLAB with just one line of code: 

\vspace{.1cm}

\begin{verbatim}
           x = INUDFT(nodes, n, b, 'tol', tol);
\end{verbatim}

\vspace{.1cm}

\noindent where \texttt{nodes} is the vector of nodes $\gamma_j = e^{-2\pi i p_j}$
defining the Vandermonde matrix, \texttt{n} is the number of frequencies (dimension of the solution vector \texttt{x}), \texttt{b} is the right hand side vector, and \texttt{tol} is a tolerance parameter that controls the accuracy of the HSS approximation (roughly equivalent to $\epsilon$). Our package also includes a factorization command that creates a sparse representation of the URV factors, which can then be applied efficiently to blocks of right-hand sides via \cref{alg:solve} with a specialized solve command. 

\section{Numerical experiments} \label{sec:experiments}
We compare the performance of the proposed ADI-based inverse NUDFT solver with several other solvers.  As our experiments reflect, the optimal method for a given problem will depend on the conditioning of $V$ (which in turn is controlled by the distribution of its nodes) as well as the number of right-hand sides, and the size of the problem. A summary for practitioners is as follows: 
\boxedtext{
    \textbf{Which solver should you use?}
    In our experiments, we find that ADI-based inverse NUDFT solver is consistently fast and accurate, even for ill-conditioned problems or with difficult distributions of nodes.
    We found two scenarios when other solvers may be preferable to ours:
\begin{enumerate}
\item For single or few right-hand side problems with mildly irregular nodes, we recommend iterative methods applied to the normal equations or adjoint normal equations (see \Cref{fig:itervsdir,fig:relreserr}). 
\item For small-to-moderate-scale problems with many right-hand sides and only mildly irregular nodes, we recommend the Kircheis--Potts method \cite{kircheis2019direct}.
\end{enumerate}
For problems not meeting either of these two scenarios (i.e., moderate-to-large scale with multiple right-hand sides or any problems with highly irregular or gappy nodes), we recommend the proposed ADI-based least-squares solver or a direct Toeplitz solver on the normal equations if conditioning permits (see \Cref{fig:relreserr,fig:direct_solver_comparisons}). 
}

A couple of caveats: while we include some commentary on preconditioning, we only engage with some of the large literature on preconditioning for Toeplitz systems of linear equations~\cite{chan2007introduction}. We also do not any include commentary related to parallelization. 
We discuss some aspects of
noise in the context of signal reconstruction in \cref{sec:imaging}, but there
remains more to investigate in future work.

The MATLAB/Octave codes to reproduce all experiments is found at \url{https://github.com/heatherw3521/NUDFT}.  

In most experiments, test signals are created by generating random complex-valued coefficients (\smash{$\{x_k\}_{k=1}^n$}) in~\cref{eq:Vandermonde_sys} whose real and imaginary parts are independently
drawn from the normal distribution of zero mean and unit variance.
This is a simple model for complex-valued signals of bandwidth $n$, and hence
(when interpreted as symmetrized frequency indices $-n/2\le k < n/2$ as in
\cref{r:conventions}) bandlimit $n/2$.
We have found that scaling with a $k$-dependent decay
does not qualitatively affect the results. 
In \cref{sec:imaging} only we instead choose coefficients to give a known 1D ``image.''
We usually choose \smash{$\{b_j\}_{j = 1}^m$} so that \cref{eq:Vandermonde_sys} is consistent, except in \cref{sec:imaging} where the second test adds noise to the data.

We consider the following four schemes for generating the locations
\smash{$\{ p_j \}_{j=1}^m \subset [0,1]$},
  which are designed to lead to increasingly poorly conditioned problems.
  \cref{fig:gridsandnodes} shows examples of the resulting
	  unit circle nodes \smash{$\gamma_j = e^{-2\pi i p_j}$.}
\begin{itemize}
\item[Grid 1:]
  We set  \smash{$p_j = ((m-j+1) + \theta \psi_j)/m$}, where \smash{$\psi_j$} is drawn from a uniform distribution on $[-1, 1]$, giving a jittered
  regular grid with
$\theta$ the so-called ``jitter parameter". When $0 \leq \theta < 1/4$,
  $V$ is proven to be well-conditioned in the square case $m=n$
 
\cite{austin2017trigonometric, yu2023stability}. For our
overdetermined
experiments with $m/n \approx 2$,
we choose $\theta = 1/2$ and observe very good conditioning (\smash{$\kappa_2(V) \approx 2.0$}, independent of $n$).

\item[Grid 2:]
  We use Chebyshev points of the second kind for $(0,1)$, which are given by the formula
  $p_j = (1+\cos \frac{\pi(j-1)}{m-1} )/2$.
  These are also Clenshaw--Curtis quadrature nodes.
  
  The resulting nodes of $V$ exhibit strong clustering along the circle toward the point $z=1$,
  and the first and last nodes coincide.
  The largest gap \smash{$p_{m/2}-p_{m/2+1}$} is close to $\pi/2m$.

  The condition that this gap be narrower than half a wavelength at the
  Nyquist frequency $n/2$ leads to the condition $m/n > \pi/2$;
  when this holds we observe \smash{$\kappa_2(V)$} small
  (of order $10$, with very weak growth in $n$).
  Conversely, for $m/n < \pi/2$ we see exponentially poor conditioning. 
  
\item[Grid 3:] The \smash{$p_j$} are independently uniform random on $[0,1)$,
  then sorted in descending order.
  This leads to random gaps, with an expected largest gap of
  \smash{${\mathcal O}(m^{-1} \log m)$}.
  For $m/n\approx 2$,
  we observe \smash{$\kappa_2(V)$} large and growing with $n$,
  of order \smash{$10^4$} for $n \approx 2000$ (the largest SVD we
  have computed). The singular value spectrum appears to
  densely fill the interval \smash{$[0,\sigma_{\rm max}(V)]$} somewhat uniformly.

\item[Grid 4:] The \smash{ $p_j$} are independently uniform random on $[0, 1-\delta]$, again sorted, where
  the enforced gap size is $\delta = 8/n$, corresponding to 4 wavelengths at the Nyquist
  frequency of $n/2$.
  The gap results in significantly higher \smash{$\kappa_2(V)$} than for Grid 3,
  namely in the range \smash{$10^6$} to \smash{$10^8$} for the largest SVDs we computed. Apart from the few smallest singular values,
  the overall singular value spectrum is similar to that of Grid 3.
\end{itemize}

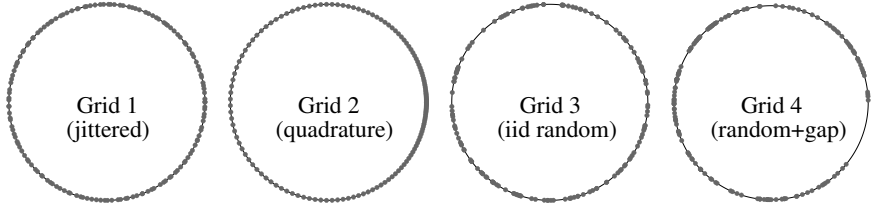
\begin{figure}[t!] 
  \centering
    \begin{minipage}{.9\textwidth} 
    \centering
    \begin{overpic}[width=.23\textwidth]{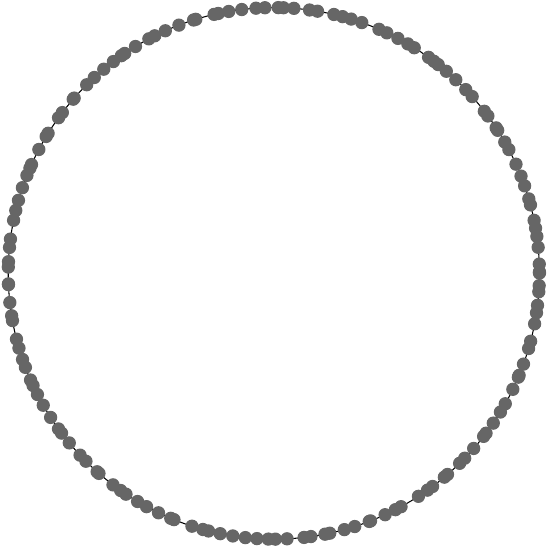} 
\put(35,45){\small{Grid 1}}   
\put(30,33){\small{(jittered)}} 
    \end{overpic}
    \hspace{.1cm}
    \begin{overpic}[width=.23\textwidth]{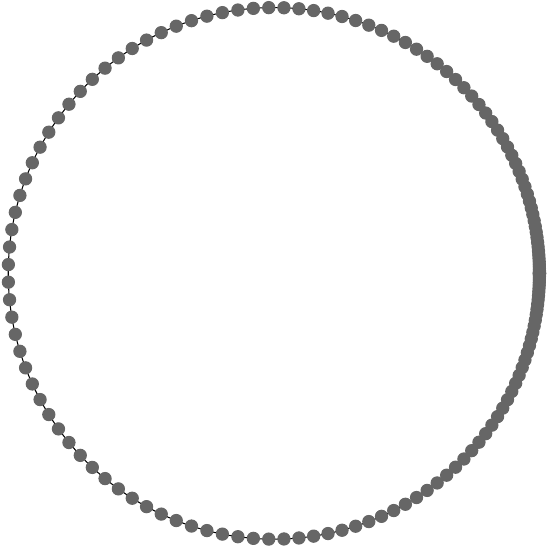} 
    \put(35,45){\small{Grid 2}}   
\put(24,33){\small{(quadrature)}}  
    \end{overpic}    
    \hspace{.1cm}
    \begin{overpic}[width=.23\textwidth]{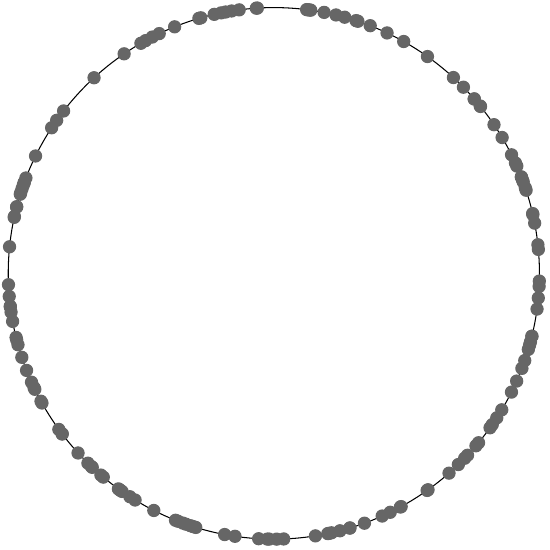} 
    \put(35,45){\small{Grid 3}}   
\put(24,33){\small{(iid random)}}       
    \end{overpic}
    \hspace{.1cm}
    \begin{overpic}[width=.23\textwidth]{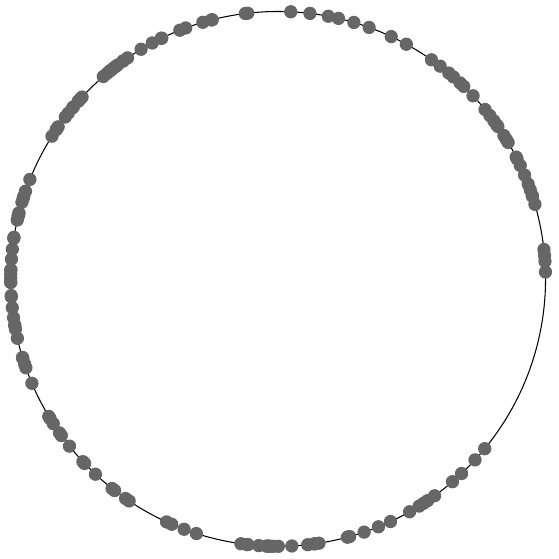} 
    \put(35,45){\small{Grid 4}}   
\put(20,33){\small{(random+gap)}}    
    \end{overpic}
      \vspace{.1cm}
  \end{minipage}
    \caption{Node sets \smash{$\{\gamma_j\}_{j=1}^m$} from the four kinds of
      nonuniform distributions tested in \cref{sec:experiments}.
      The first two give \smash{$\kappa_2(V)={\mathcal O}(1)$}, while for the last two
      \smash{$\kappa_2(V)$} is large ($ \gg 10^4$ in large problems).
      For visibility, a very small mode number $m = 150$ is shown. In Grid 4
      the gap is four wavelengths at the Nyquist frequency $n/2$, choosing $n=m/2$.}
  \label{fig:gridsandnodes}
\end{figure}

\subsection{Iterative methods vs.\ ADI-based direct solver} \label{sec:iterative} 
We compare timings\footnote{For experiments in this section, all methods were tested using a Dell laptop computer with a 12th Gen Intel Core i7-1280P 1.8 GHz processor and 32 GB RAM.} and relative residual error across several iterative methods in \Cref{fig:itervsdir}.
The algorithms are implemented in MATLAB R2023a,
where each iteration uses either MATLAB's FFT library, or the MATLAB
wrapper for the state of the art FINUFFT library \cite{barnett2019parallel}
for NUDFT transforms (multiplications by $V$ or \smash{$V^*$}).
We compare against five iterative methods we call CG nor, PCG nor Strang, FP adj sinc, CG adj, and PCG adj sinc; see \cref{sec:iterative-description} for a description of these methods.

The performance of these methods is compared to the proposed
ADI-based least-squares direct solver in \Cref{fig:itervsdir},
for a large problem of size $524, 288 \times 262, 144$
(oversampling ratio $m/n=2$).
The iterative methods terminate by either achieving a relative residual of \smash{$10^{-7}$} or performing a maximum number of iterations ($8,000$ for methods on adjoint equations, $10,000$ for methods on normal equations).

The convergence of the iterative methods depends in general on the
singular value distribution of $V$ and on the right-hand side
\cite{axelsson00,carson24};
a well-known bound on its geometric rate is controlled by
the condition number \smash{$\kappa_2(V)$} \cite{saadbook}.
The latter diverges as the nodes become increasingly irregular.
Grids 1 and 2 do not cause much trouble (iteration counts are in the range 10--100), but all of the iterative methods struggle when when Grids 3 and 4 are used. For example, CG on the normal equations  requires 5734 and 5416 iterations, respectively, and CG on the adjoint normal equations requires 7788 and 7769 iterations, respectively. 
The two types of preconditioning tested do not seem to help here.
In contrast, the ADI-based least-squares solver is agnostic to the grid choice, and always takes around 22 seconds to execute for this problem size.
Here, the HSS rank parameter is set as \smash{$\epsilon = 10^{-10}$}.

\begin{figure}[t!]  
  \centering
    \begin{minipage}{.44\textwidth} 
    \centering
    \begin{overpic}[width=\textwidth]{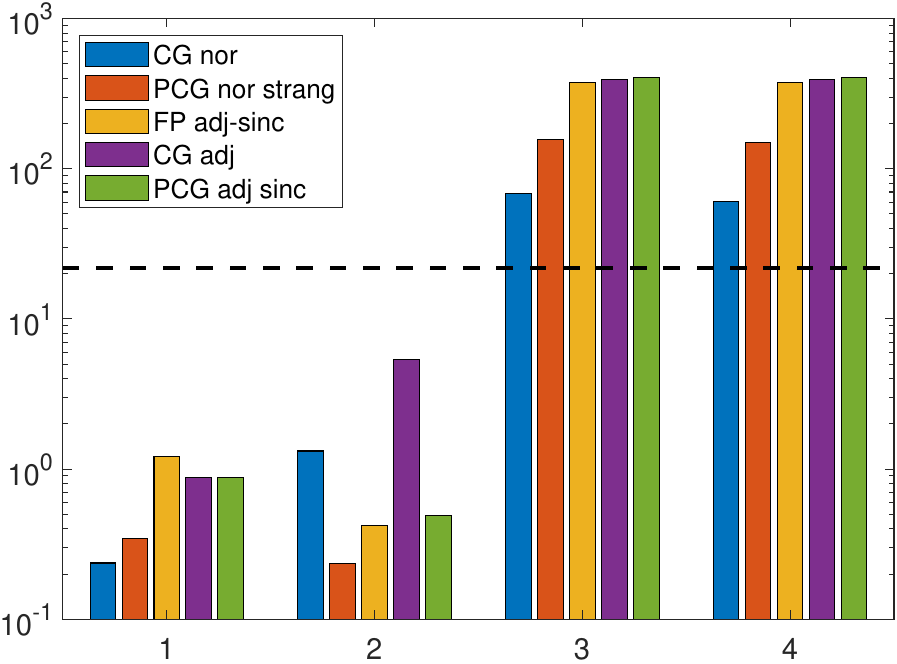} 
    \put(-5,29){\rotatebox{90}{\small{\textit{seconds}}}}
    \put(40,-5){\small{\textit{Grid choice}}}
     \put(12,41){\tiny{ADI-based LS ($\approx 22$ sec.)}}
    \end{overpic}
  \end{minipage}
  \hspace{.5cm}
  \begin{minipage}{.44\textwidth} 
    \centering
    \begin{overpic}[width=\textwidth]{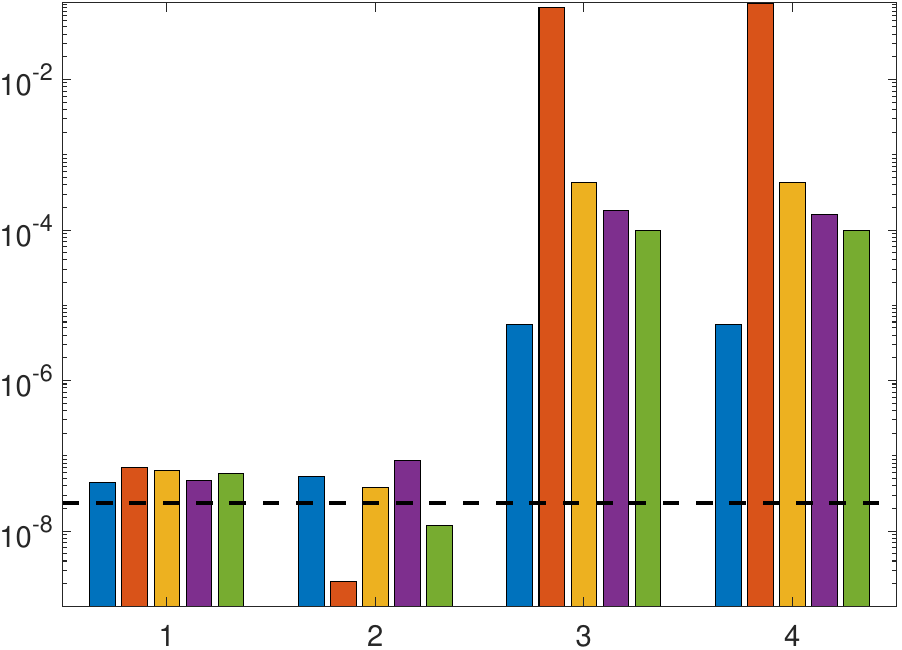}
       \put(-5,18){\rotatebox{90}{\small{\textit{relative residual}}}}
       \put(100,16){\rotatebox{60}{\tiny{ADI-based LS}}}
        \put(40,-5){\small{\textit{Grid choice}}}
    \end{overpic}
  \end{minipage}
  \caption{The bar plots compare the performances of various iterative methods with the
    proposed
    ADI-based least-squares solver for solving $Vx = b$ for the four different node sets (grids) of \cref{fig:gridsandnodes}.
    $V$ is of size $m\times n = 524, 288 \times 262, 144$.
    \emph{Left:} The wall clock time in seconds achieved by each method is plotted on a log scale for each grid choice. The ADI-based LS method (dotted black line) is insensitive to the grid choice and performs similarly (about 22 seconds) for all problems of this size. \emph{Right:} The relative 2-norm
    residual (data error) achieved by each method is plotted on a log scale for each grid choice. The ADI-based least-squares method achieves roughly the same relative residual for every problem. The HSS rank parameter is \smash{$ \epsilon = 10^{-10}$}.
    In most of the iterative methods for grids 3 and 4, a maximum
    iteration count around \smash{$10^4$} was reached (see \cref{sec:iterative}).
  }
  \label{fig:itervsdir}
\end{figure}
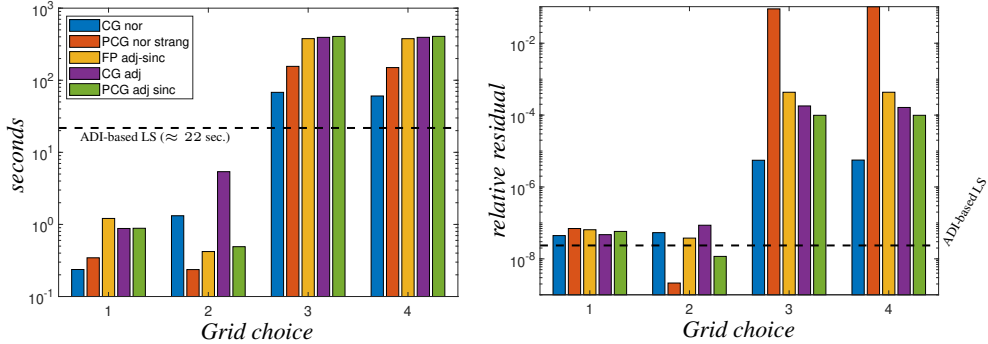

An informative comparison of the methods is given in \cref{fig:relreserr}, where we plot the relative residual against the execution time per right-hand-side
required to achieve it\footnote{We omit CG with Strang preconditioning from these results as it performed much more poorly than other methods.}.   Unlike the iterative methods, the ADI-based least-squares solver gains efficiency as the number of right-hand sides is increased. For this test, $V$ is of size $29,492 \times 16,384$
(so oversampling $m/n\approx 1.8$),
with Grid 3 used for sample locations. Various stopping residuals are explored for the iterative methods%
\footnote{Interestingly, the residual error for the iterative methods appears to decay only algebraically (rather than geometrically) with iteration count. This type of {\em sublinear} CG convergence has been observed
\cite[Fig.~1]{carson24} and analyzed \cite[Sec.~3.2]{axelsson00}
when the eigenvalue spectrum densely fills an interval that includes zero,
which seems to be the case for Grid 3. We leave a full understanding for future work.}.

One conclusion we draw is that if 20 or more RHS are needed, the ADI-based
solver beats all iterative methods, even at low accuracy.

\subsection{Comparisons with direct methods} \label{sec:Direct}
Here, we compare against some alternative direct solvers, including the method from Kircheis and Potts in~\cite{kircheis2019direct} and an ADI-based Toeplitz solver we have implemented and applied to solve the normal equations. Overall, we find that the specialized ADI-based solvers are much faster and more robust than other direct solvers currently available. For those familiar with hierarchical solvers, it may be surprising that the ADI-based methods perform as well as they do on small- and moderate-sized problems. The computational and memory costs for hierarchical methods often involve large constants, making them useful primarily in the large-scale setting. Our HSS construction method is leaner than many hierarchical solvers because we heavily leverage the displacement structure of $V$ (e.g., to avoid forming or storing $V$ or $C$ outright, to supply explicit a priori bounds for low rank approximations, and to apply an unusually fast low rank approximation method).

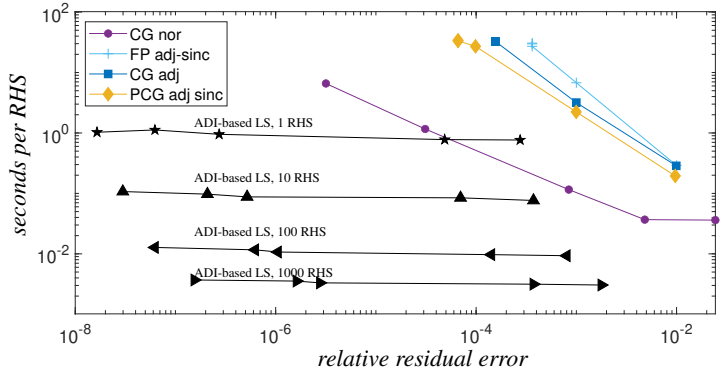
\begin{figure}[t!]   
  \centering
  \begin{minipage}{.70\textwidth} 
    \centering
    \begin{overpic}[width=\textwidth]{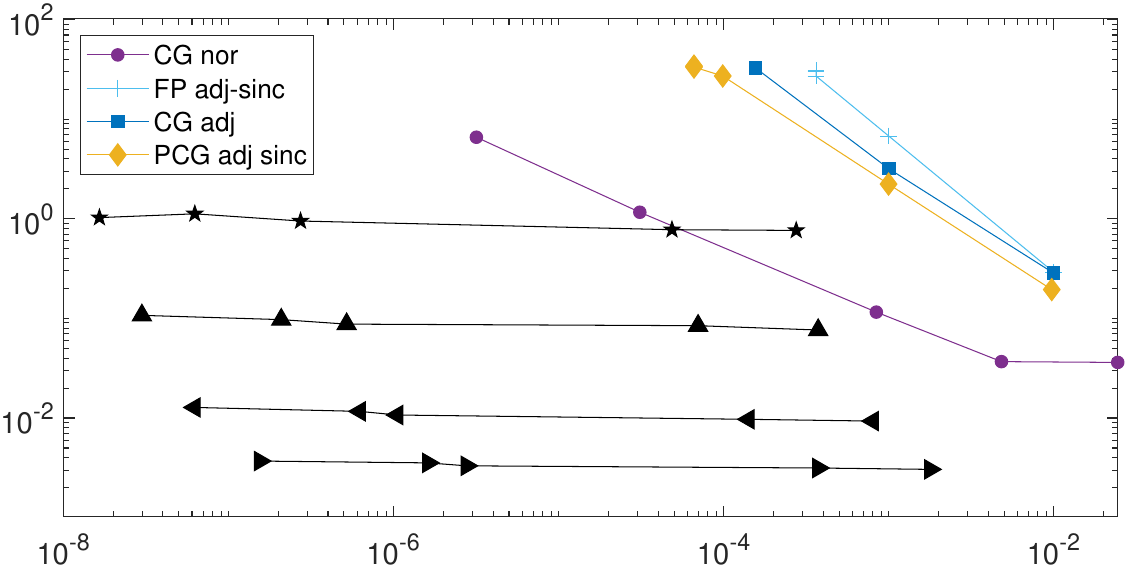}
     \put(-4,16){\rotatebox{90}{\small{\textit{seconds per RHS}}}}
    \put(41,-3){\rotatebox{0}{\small{\textit{relative residual error}}}}
     \put(23,32){\tiny{ADI-based LS, 1 RHS}}
     \put(23,24){\tiny{ADI-based LS, 10 RHS}}
     \put(23,16){\tiny{ADI-based LS, 100 RHS}}
     \put(23,10){\tiny{ADI-based LS, 1000 RHS}}
    \end{overpic}
  \end{minipage}
  \vspace{1ex}
  \caption{The performances of several methods are compared by plotting the relative residual error achieved against the execution time required to achieve it on a log-log scale. Execution times indicate the time per right-hand side,
    which for the the iterative methods only depends on accuracy (residual).
    The proposed ADI-based direct method has a run-time almost
    independent of accuracy, and shows an efficiency per RHS growing
    linearly up to several hundred RHS.
    In these problems, Grid 3 is used to generate the $29,492 \times 16,384$ matrix $V$.} 
  \label{fig:relreserr}
\end{figure}

In the left panel of \cref{fig:direct_solver_comparisons}, we compare the performance of our method against the Kircheis--Potts method~\cite{kircheis2019direct} for a single right-hand side\footnote{This test was performed on a 2013 Mac Pro with a 2.7 GHz 12-Core Intel Xeon E5 processor and 64 GB 1866 MHz DDR3 SDRAM}. The Kircheis--Potts method is designed to be effective in settings where the nodes are not too irregularly clustered and there are many right-hand sides. The method has a \smash{$\mathcal{O}(m^2 +n^2)$} precomputation step that can be prohibitive in the large-scale setting. We also compare against dense direct solution (MATLAB's \texttt{mldivide}, i.e., ``backslash''), as well as the fast ADI-based Toeplitz solver from~\cite[Ch.~4]{wilber2021computing} applied to the normal equations \cref{nor}. The latter is the natural Toeplitz analog of our direct least-squares method for Vandermonde matrices. As shown in~\cite[Ch.~4]{wilber2021computing}, the performance of this Toeplitz solver is comparable to related implementations~\cite{xia2012superfast} that use randomized numerical linear algebra for HSS construction. 

For all of these methods, the cost for solving the system with multiple right-hand sides should be considered in stages. In the first stage, a factorization is computed and stored. For our solvers, this is the hierarchical ULV/URV factorization. For backslash, one could store the factors from QR decomposition. For the Kircheis--Potts method, the initial factorization involves constructing a sparse and diagonal matrix that, along with a DFT matrix, approximately diagonalize the system.  The second stage involves applying these factors to the right-hand sides. We supply comparisons for the two fastest methods in \cref{fig:direct_solver_comparisons}. Kircheis--Potts requires the application of a sparse matrix-vector product plus and an FFT for each right-hand side. Our method also includes an FFT for each right-hand side (step 4 of \Cref{alg:generalsolver}), plus the hierarchical unitary multiple-triangular solve step in \cref{sec:solvestep}. 
The solve step applied to multiple right-hand sides involves many dense matrix-matrix products. Since these are handled at the BLAS3 level, we see tremendous practical speedup per right-hand-side at this stage of the computation.

As we see in the left panel of \cref{fig:direct_solver_comparisons}, the  hierarchical Toeplitz and Vandermonde solvers are the fastest over the other solvers by a factor of 100 when accounting for both stages. For small values of $n$, our Vandermonde solver outperforms the Toeplitz solver by a small amount, transitioning to slightly underperforming the Toeplitz solver for large values of $n$. We believe that this performance inversion is due to the Vandermonde solver's somewhat higher memory requirements.
The right panel of \cref{fig:direct_solver_comparisons} shows that the second stage of the Kircheis--Potts solver is about $20\times$ faster than that of the hierarchical Vandermonde solver, showing that the former is very effective for medium-sized problems where the precomputation cost is affordable and the distribution of nodes is not too irregular.

Note that, for this test using Grid 1 (jittered), CG nor was the best iterative
method (verified on the left side of \cref{fig:itervsdir}).
Thus we did not include the other four iterative methods in these tests.

We also do not include comparisons against \smash{$\mathcal{O}(n^2)$} Toeplitz solvers that exploit the Gohberg--Semen\c{c}ul formulae, since in initial
tests using the NFFT implementation \cite{keiner2009using} we found it uncompetitive even at small scale, having a runtime exceeding that of the Kircheis--Potts method
 \cite{kircheis2019direct}.

\begin{figure}[t!] 
  \centering
  \begin{minipage}{.48\textwidth} 
    \centering
    \begin{overpic}[width=\textwidth]{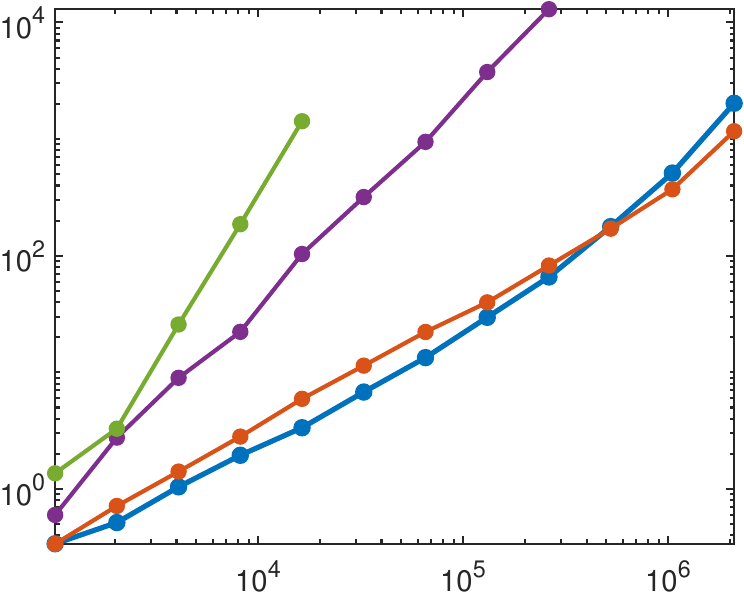}
      \put(21,42){\rotatebox{56}{\small{backslash}}}
      \put(35,45){\rotatebox{44}{\small{Kircheis--Potts}}}
      \put(30,23){\rotatebox{30}{\small{nor. eqns + ADI-Toeplitz}}}
      \put(33,13){\rotatebox{29}{\small{ADI-based LS solver}}}
      \put(53, -3){\small{\textit{n}}}
      \put(-4, 18){\rotatebox{90}{\small{\textit{execution time (seconds)}}}}
    \end{overpic}
  \end{minipage}
  \hspace{.2cm}
  \begin{minipage}{.48\textwidth} 
    \centering
    \begin{overpic}[width=\textwidth]{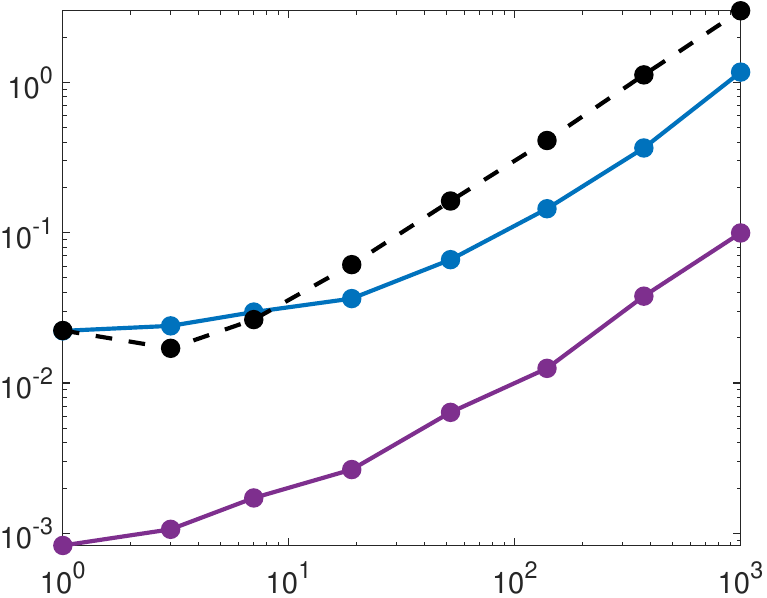}
      \put(32, -5){\small{\textit{\# of right-hand sides}}}
      \put(33,25){\rotatebox{28}{\small{ADI-based LS (solve-step only)}}}
      \put(38,7){\rotatebox{28}{\small{Kircheis--Potts (solve-step only)}}}
      \put(40,44){\rotatebox{35}{\small{CG nor}}}
    \end{overpic}
  \end{minipage}
  \vspace{1ex}
  \caption{\emph{Left:} The wall clock time required to solve $Vx =b$, where $V$ is a NUDFT matrix of dimension $2n \times n$ is plotted against $n$ on a log-log scale.
    We compare the proposed ADI-based LS solver (red) with several other direct methods, including MATLAB's dense backslash (green), the overdetermined solver from Kircheis and Potts in~\cite{kircheis2019direct} (purple), and the fast direct ADI-based Toeplitz solver~\cite[Ch. 4]{wilber2021computing} applied to the normal equations (red). and the HSS rank parameter is set to \smash{$\epsilon = 10^{-10}$} in the ADI-based solvers.
 Grid 1 was used, so $V$ is well-conditioned (this gives CG nor a strong advantage!).
 \emph{Right:} The execution time in seconds is plotted against the number of right-hand sides in the problem $VX = B$ on a log-log scale for three different solvers, where $V$ is of size $16384 \times 8192$. We exclude factorization steps from the timings.
 The well-conditioned $V$ is a best-case scenario for CG, which gets slower as $V$ becomes more ill-conditioned. The performance of the ADI-based LS method does not depend on the conditioning of $V$.}
  \label{fig:direct_solver_comparisons}
\end{figure}
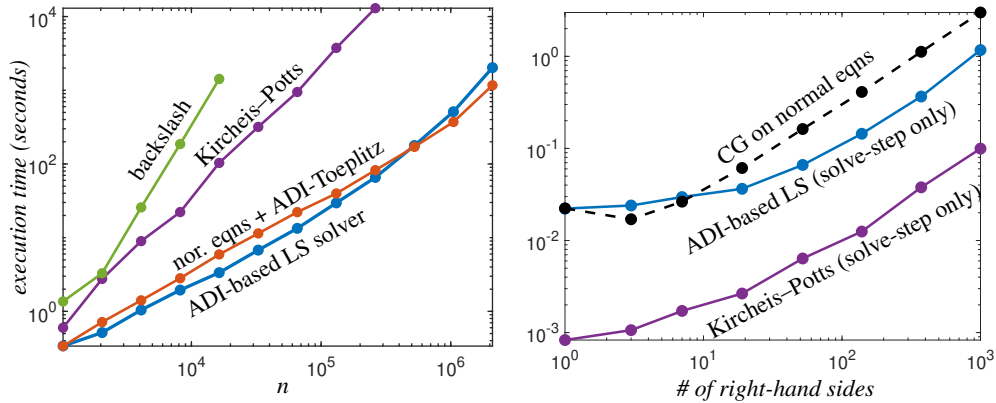

\subsection{An ``image'' reconstruction example} \label{sec:imaging}
One of the major applications for non-uniform discrete Fourier transforms is signal reconstruction problems.
In our previous experiments, we have considered the ideal case where the signal has little structure (all Fourier coefficients with the same variance),
while the measurements \smash{$b_j$} are free of noise.
In practice, measurements are always corrupted with noise,
so that $Vx=b$ becomes inconsistent.
Yet a true least-squares solution as in \cref{eq:Vandermonde_sys} can
effectively average this noise if $V$ is well conditioned.
If not, noise will usually be grossly amplified, demanding
regularization.  We give some discussion on this topic in~\cref{sec:extensions}, but it is an area for future work. 

For now, we make some preliminary observations using a toy 1D reconstruction task. In \cref{fig:Imaging}, we choose a discontinuous signal function $f(2\pi p)$ on $0\le p < 1$, where
\[
f(x) = \left\{\begin{array}{ll}
\sin(8x^2),& x<3\\
1,& \frac{7}{2}<x<\frac{9}{2}\\
\max(0,1-2|x-\frac{11}{2}|), & \mbox{otherwise}
\end{array}
\right.
\]
contains an oscillatory chirp, a top-hat and a triangle part.
Its exact Fourier coefficients $$c_k = \int_0^1 e^{2\pi i kp} f(2\pi p) \mathrm{d}p$$
for $-n/2\le k < n/2$
are approximated using equispaced quadrature on $10n$ nodes.
We then set \smash{$x_k = \exp(-0.5(7k/n)^2) c_k$} which low-pass filters the signal
to make it approximately bandlimited (here the Gaussian factor is about
0.002 at the maximum frequency $\pm n/2$). These $\{x_k\}_{-n/2}^{n/2-1}$
are now are taken as the true Fourier coefficients.
For a given sample grid,
$b_j = \sum_{k=-n/2}^{n/2-1} \gamma_j^k x_k$ is evaluated using a type-II. For these experiments, we choose $n = 2048$ and take $m = 3687$ samples of the signal, so $m/n \approx 1.8$. 
In order to apply NUDFT type-II solvers as defined in the rest of this paper
with frequencies $0,\dots,n-1$ offset by $n/2$ from this,
prephased data \smash{$\gamma_j^{n/2} b_j$} must be used (\cref{r:conventions}).
We solve the least-squares inverse type-II NUDFT problem using two iterative methods (CG on the normal and the adjoint normal equations) and one direct method (ADI-based least-squares solver). This constructs an approximation to the coefficients in~\cref{eq:Vandermonde_sys},
from which the signal (Fourier series) is plotted
on a denser grid of $6n$ equally-spaced points using a padded plain FFT.

We consider two problematic settings. First, we use Grid 4, which 
results in multiple sizeable gaps between sampling points (one forced, others due to random uniform sampling). This corresponds to large condition number, with \smash{$\kappa_2(V) \approx 10^7$}. 
Thus, one would not be able to recover the ground truth in a realistic (noisy) setting.
However, with noise-free data we see that the reconstruction via the
ADI-based direct solver is visually indistinguishable from the true signal,
while the two iterative solvers have $\bigO(1)$ errors at certain problematic locations
(large gaps).
This reflects the smaller residual error that the direct solver achieves.

The upper right panel of \Cref{fig:Imaging} is a zoomed-in display of a small portion of these reconstructions that clearly illustrates these behaviors. While we do not display results for other direct solvers (e.g., the Toeplitz direct solver on the normal equations and backslash in MATLAB), we observe that reconstructions from these methods look similar to the ADI-based result.
This illustrates an accuracy advantage that fast direct solvers can have in ill-conditioned settings.

In the lower left panel of \Cref{fig:Imaging}, we consider another problematic setting. Here, we sample the signal using Grid 1 (jittered), so that $V$ is well-conditioned, but then add iid Gaussian noise of standard deviation \smash{$10^{-2}$} to both the real and complex parts of the observations (visible as blue dots in lower-right panel), to give the right-hand side vector in $Vx = b$.
The oversampling helps
average away (reduce) some noise in the solutions from CG on the normal equations and the ADI-based direct solver.
Moreover these two solutions are extremely close, differing by about \smash{$10^{-5}$} in
the \smash{$L^2$}-norm.
In contrast, the CG on adjoint normal equations is visibly much noisier.

\begin{figure}[t!] 
  \centering
  \begin{minipage}{.74\textwidth} 
    \centering
    \begin{overpic}[width = \textwidth]{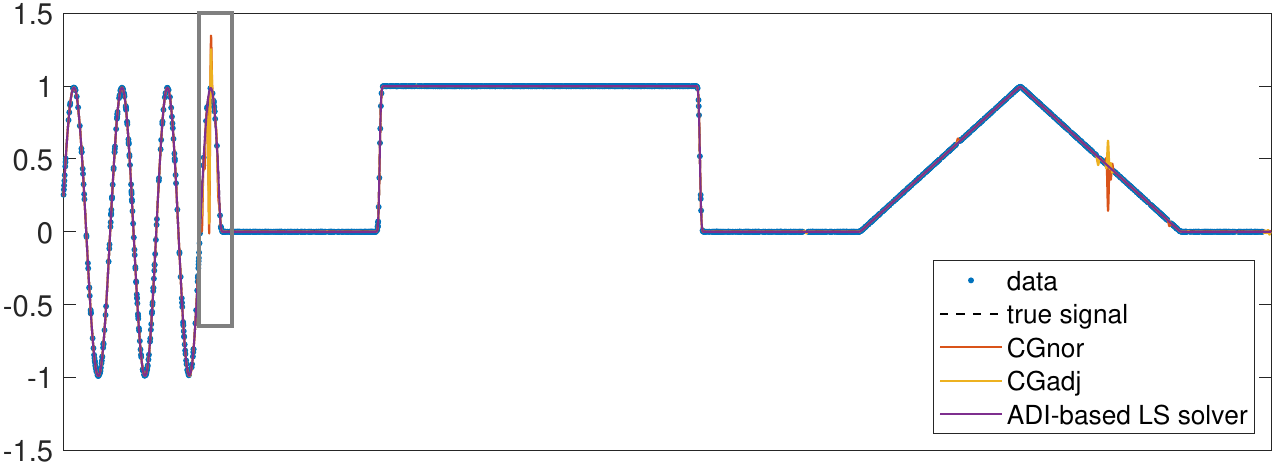}
   \put(20, 38){\small{noise-free samples from a gappy sampling (Grid 4)}}
    \end{overpic}
  \end{minipage}
  \begin{minipage}{.20\textwidth} 
    \centering
    \begin{overpic}[width=\textwidth]{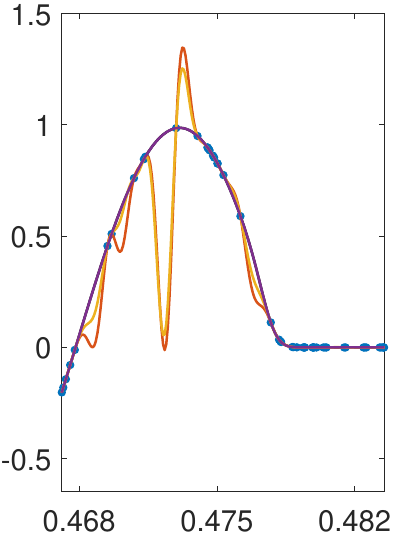}
    \end{overpic}
    \end{minipage}
    
\vspace{.22cm}

      \begin{minipage}{.74\textwidth} 
    \centering
    \begin{overpic}[width = \textwidth]{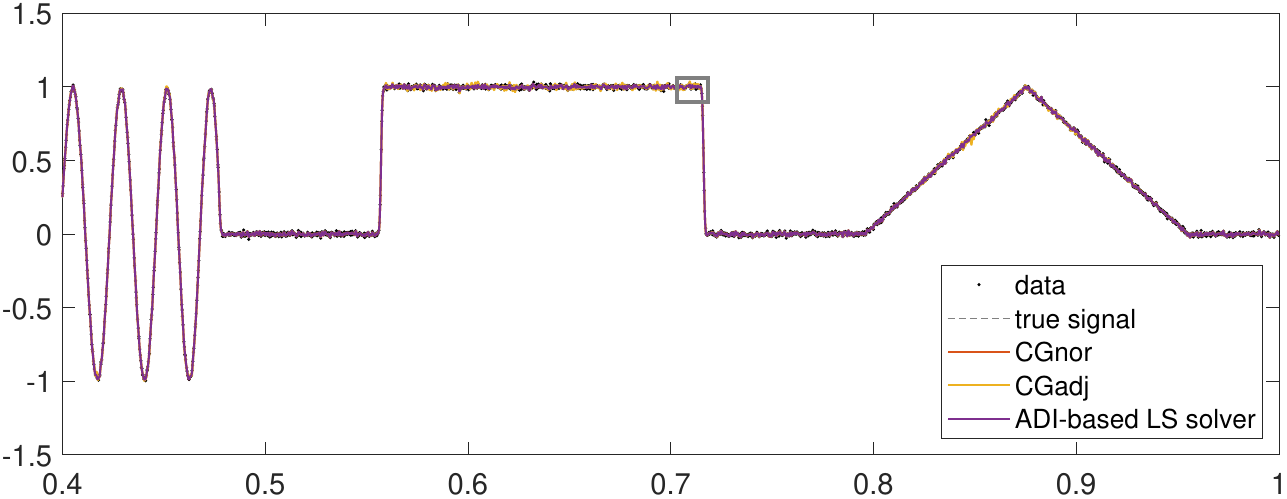}
    \put(21, 39){\small{noisy samples from a jittered sampling (Grid 1)}}
    \end{overpic}
  \end{minipage}
  \begin{minipage}{.20\textwidth} 
    \centering
    \begin{overpic}[width=\textwidth]{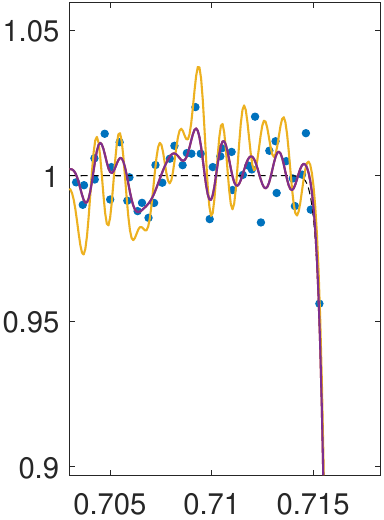}
    \end{overpic}
  \end{minipage}
  \caption{1D ``image'' reconstruction example with $n=2048$ coefficients and $m=3687$.  
     Samples are shown only on  a subset of the domain where \smash{$0.4<p_m<1$}.
      \emph{Upper:} The real part of several reconstructions of a signal are plotted on a fine regular grid, along with the true underlying signal (obscured as it closely coincides with the ADI-based reconstruction). The blue dots indicate samples.
      The sample locations are given by Grid 4 (iid random with gap). The grey rectangle highlights a portion of the signal that is zoomed in on to the right.
      \emph{Lower:} Reconstructions for the same signal but using Grid 1 (jittered),
      with data corrupted by additive noise.  The grey rectangle highlights a portion of the signal zoomed in on to the right,
      where one sees that the CG nor and ADI-based reconstructions visually coincide.}
  \label{fig:Imaging}
  \end{figure}

\section{Extensions} \label{sec:extensions}
Various extensions and future investigations arise from this work.
Some extensions are immediate, while some are more involved.
We anticipate that the proposed method can be extended quite
easily in the following three directions.

1) Regularization of \cref{eq:least_squares} by addition of a term
$\lambda x^* L x$ is crucial in various ill-conditioned inverse problems.
At least in the case $L=I$,
we believe that recasting as an expanded $(m+n)\times(m+n)$ system
followed by a block perfect shuffle would allow the use of a fast direct
ULV HSS solver \cite{chandrasekaran2006fast}.

2) Inverting the type-I NUDFT (see \cref{r:types} for its interpretation)
has applications including superresolution imaging \cite{batenkov2019spectral}.
  Since the system matrix is now $V^*$,
  one may exploit that $\overline{F} V^*$ is Cauchy-like
  and use similar ADI-based strategies to approximate it with an HSS matrix.
In the overdetermined case $n \geq m$, the solver in \cref{sec:solver} could be applied
directly.  Underdetermined systems with regularization are also of practical interest, but more work is required to adapt our methods to that setting. A semidirect method for the underdetermined case is included in~\cite{ho2014fast}.

3) Our technique can be applied to solve $Xy=f$ for any matrix $X$ that satisfies a Sylvester matrix equation $AX-XB = F$, where (i) the rank of $F$ is small (e.g., 1 or 2), and (ii) the matrices $A$ and $B$ can be diagonalized via fast transforms. This includes Toeplitz, Hankel, and Toeplitz+Hankel matrices, as well as nonuniformly sampled Chebyshev--Vandermonde matrices~\cite{koev1999matrices}.

\subsection{The 2D setting} A more challenging question is how the method can be extended to the 2D and 3D cases, which are arguably the most important for applications.  The type-II 2D NUDFT considers $m$ ordered pairs of sample locations,  \smash{$\{ (t_1, s_1), \ldots (t_m, s_m)\} \subset [0, 1]^2$}, with corresponding samples \smash{$\{b_1, \ldots, b_m\}$}. The goal is to recover the coefficient matrix $X$ in
\begin{equation} \label{eq:2Dprob}
  b_j = \sum_{k = 1}^{n_1} \sum_{k' = 1}^{n_2}
  e^{-2 \pi i \left((k-1) t_{j} + (k'-1)s_{j} \right)} X_{kk'}, \qquad 1 \leq j \leq m.
\end{equation}
If the sample locations \smash{$\{(t_j, s_j)\}$} form a tensor grid on $[0,1]^2$, then a fast 2D solver follows directly from repeated applications of the 1D solver. However, this is not the typical case. In general,  
we can express~\cref{eq:2Dprob} as an $m \times n_1n_2$ linear system involving a multivariate Vandermonde matrix $V$. One can also view $V$ as the face-splitting product~\cite{slyusar1999family} \smash{$V_\eta \bullet V_\xi$}, where \smash{$V_\eta$} is an $m \times n_1$ Vandermonde matrix, and \smash{$V_\xi$} is an $m \times n_2$ Vandermonde matrix. There are various ways that one can transform $V$ into a matrix involving Cauchy-like structures, but there is no obvious way to transform $V$ into a matrix that is well-approximated by an HSS matrix, since many of the Cauchy-like blocks are not numerically of low rank.  We are confident that in 2D a \smash{$\mathcal{O}(m^{3/2})$} solver can be achieved using nested dissection~\cite{martinsson2019fast}, with further details to be supplied in later work.

\section*{Acknowledgments}

The authors are grateful for discussions with Dima Batenkov, Leslie Greengard, Alex Townsend, and Joel Tropp, as well as support to attend a workshop at Casa Matemática Oaxaca (Banff International Research Station) where some early work was done.
We thank Joel Tropp for access to the machine used to make the left panel of \cref{fig:direct_solver_comparisons}. We also thank the anonymous referees for suggestions that improved the manuscript. The Flatiron Institute is a division of the Simons Foundation.

\appendix

\section{Description of iterative methods}
\label{sec:iterative-description}
Here is a description of the iterative methods to which we compare in \cref{sec:iterative}:

\begin{enumerate}
\item \textbf{CG nor.} The conjugate gradient method is applied to the normal equations
\be
V^*V x = V^* b,
\label{nor}
\ee
where the Toeplitz matrix
with entries \smash{$(V^* V)_{kk'} = \sum_{j=1}^m \gamma_j^{k'-k}$}
is applied via discrete nonperiodic convolution using the padded%
\footnote{We pad to $2n$ since it is much more efficient for FFTs
than the minimum padding length of $2n-1$.}
FFT.
Here, the Toeplitz vector is
precomputed by applying the type-I NUDFT to the vector of ones.
CG on the normal equations minimizes the residual 2-norm
as in \eqref{eq:least_squares}, and is thus also known as CGNR
\cite[\S8.3]{saadbook}.

\item \textbf{PCG nor Strang.} We apply Strang
  preconditioning~\cite[\S2.1]{chan2007introduction} to the previous
  CG method for
  \cref{nor}:
this uses a $n\times n$ circulant
preconditioner constructed from the central $n$ diagonals of the Toeplitz
matrix. Such a preconditioner is generally effective for
diagonally-dominant matrices, which does not
generally hold for $V^*V$.

\item \textbf{FP adj sinc.}
We include a fixed-point iteration
found to be effective in the MRI setting by Inati et al.\ \cite{inati07draft}
to solve the adjoint
(``second kind'') $m\times m$ normal equations
\be
V V^* z = b,
\label{adj}
\ee
returning \smash{$x=V^*z$}. Assuming that $W$ is an approximate inverse of \smash{$VV^*$},
then from the current guess $z_k$ the residual
\smash{$r_k = VV^* z_k - b$} is computed,
from which the next iterate is
\smash{$z_{k+1} = z_k - Wr_k$}. This iteration starts with \smash{$z_0=b$},
and stops when \smash{$\|r_k\|\le\epsilon\|b\|$}.
Turning to $W$, a good diagonal choice is
given by the Frobenius norm minimization
\be
w = \argmin_{w\in \R^m, \, W = \diag\{w\}} \|I - WVV^*\|^2_{\rm F},
\label{wfrob}
\ee
motivated by the fact that the convergence rate of the
above iteration is bounded by the spectral norm \smash{$\|I - WVV^*\|$},
yet the Frobenius norm is much simpler to minimize than the spectral norm.
Specifically, the exact solution to \cref{wfrob}
may easily be found to be the have elements
\be
w_j = \frac{(VV^*)_{jj}}{(VV^*VV^*)_{jj}}
=
\frac{n}{\sum_{j'=1}^m|(VV^*)_{jj'}|^2}, \qquad j=1,\dots,m,
\label{wopt}
\ee
which may be interpreted as the discrete analog
of ``sinc$^2$ weights'' in the MRI literature
\cite{choimunson,greengard2006fast},
since the Dirichlet kernel entries obey
$$\lim_{n\to\infty} (VV^*)_{jj'} = n \sinc n\pi (p_j-p_{j'});$$
also see \cite[Rmk.~3.16]{kircheis2023direct}.
An efficient computation of $w$
groups terms along each diagonal $q=k-k'$ as follows:
\[
(VV^*VV^*)_{jj} =
\sum_{k,k'=1}^n \gamma_j^{k-k'} \sum_{j'=1}^m \gamma_{j'}^{k'-k}
=
\sum_{|q|<n} \gamma_j^q \cdot (n\!-\!|q|) \sum_{j'=1}^m \overline{\gamma_{j'}}^{q},
\]
which can now be evaluated by a
type-I NUDFT with unit strengths, followed by a type-II
(each transform being double-sized with shifted frequency range).
The triangular weights $(n\!-\!|q|)$ show that the
denominator in \cref{wopt} is a {\em Fej\'er kernel}.

\item \textbf{CG adj.}
This applies CG to \cref{adj}, 
where the matrix \smash{$VV^*$} is, as above,
applied by a type-I and then type-II transform.
\smash{$x=V^*z$} is then returned.
In this case CG minimizes the 2-norm of
the error in $x$, as opposed to the residual,
and is also known as CGNE or Craig's method
\cite[\S8.3]{saadbook}.

\item \textbf{PCG adj sinc.}
  This uses the diagonal matrix $W=\diag\{w\}$
  with the optimal (``sinc$^2$'' or Fej\'er)
  weights \cref{wopt} as the preconditioner for CG on \cref{adj},
  as a variant of the previous method.
  This is discussed in \cite[Rmk.~1]{greengard2006fast}.
\end{enumerate}

\section{The ADI method for low rank approximations to Cauchy matrices} \label{sec:fadi_alg}
Simplifying our notation from~\eqref{eq:dispY}, consider a Cauchy matrix \smash{$C \in \mathbb{C}^{p \times q}$} that satisfies 
\begin{equation}
\label{eq:basicdisp}
 \Gamma C - C \Lambda = uv^*,
\end{equation}
with \smash{$u \in \mathbb{C}^{p \times 1}$}, \smash{$v \in \mathbb{C}^{q \times 1}$}, and  \smash{$\Gamma \in \mathbb{C}^{p \times p}$} and \smash{$\Lambda  \in \mathbb{C}^{q \times q}$} diagonal with spectral sets contained in disjoint arcs on the unit circle as in \Cref{lemma:arcs}.  After $k$ steps of the fADI algorithm, one produces \smash{$Z \in \mathbb{C}^{p \times k}$},  \smash{$W \in \mathbb{C}^{q \times k}$,} so that \smash{$C^{(k)} = ZW^*$} is an approximation to $C$. Pseudocode for fADI in this setting is given in \cref{alg:fadicauchy}. 

\begin{algorithm}[h!]
  \caption{factored ADI for the low rank approximation of a Cauchy-like matrix} \label{alg:fadicauchy}
  \begin{algorithmic}[1]
    \Procedure{factored ADI}{$C, \Gamma, \Lambda, u, v, \epsilon$} 
     \State  Use the bounds in \Cref{thm:cHSS_col} to determine $k$ such that $\|C - C^{(k)}\|_2 \leq \epsilon \|C\|_2$. 
    \State Compute ADI shift parameters $\{\alpha_j, \beta_j\}_{j = 1}^k$  using~\cite[Cor.~2] {wilber2021computing}.  
\State $Z_1 = (\beta_1-\alpha_1) (\Gamma - \beta_1 I)^{-1}u$
\State$W_1 = (\Lambda^*-\overline{\alpha}_1 I)^{-1}v$

    \For{$ j = 1$ to $k\!-\!1$}
	\State $Z_{j +1} = (\beta_{j+1}-\alpha_{j+1}) (\Gamma - \alpha_jI) (\Gamma-\beta_{j+1} I)^{-1} Z_j$
	\State $W_{j+1} = (\Lambda^* - \overline{\beta}_j I) (\Lambda^* - \overline{\alpha}_{j+1} I)^{-1} W_j$
    \EndFor
\State $Z = \left[ Z_1, Z_2, \cdots, Z_k \right]$
\State $W = \left[ W_1, W_2, \cdots, W_k \right]$
   \EndProcedure
 \end{algorithmic}
\end{algorithm}

Note that, since the matrices $\Gamma$ and $\Lambda$ are diagonal, the application of the shifted products and inverses to a column vector is just an entrywise product with column vectors. The optimal shift parameters can be computed at a trivial cost using elliptic integrals. Code for this can be found in our GitHub repository: \url{https://github.com/heatherw3521/NUDFT}. \Cref{alg:fadicauchy} can be used to directly construct the low rank factors of HODLR blocks in the NUDFT matrix. However, for  low rank approximations to the HSS rows and columns, the blocks are of a shape where either $q \gg p$ (HSS rows) or $p \gg q$ (HSS columns). Without loss of generality, assume it is the case that $q \gg p$. Then, it is preferable to avoid constructing $W$ altogether.  Instead, as in \cref{sec:constr-gener}, we construct $Z$ only and then apply a QR decomposition to \smash{$Z^*$}. From this, one can construct a one-sided interpolative decomposition as in~\cite{cheng2005compression} without ever sampling $C$ or applying operations with a cost that involves $q$. 

\vspace{2cm}

\bibliographystyle{siam}
\bibliography{refs}

\end{document}